\documentclass[twocolumn,10pt]{IEEEtran}
\DeclareUnicodeCharacter{3000}{ }
\DeclareUnicodeCharacter{FF1B}{~}
\usepackage[bottom=0.9in,top=0.6in,left=0.625in,right=0.625in]{geometry}

\usepackage{ifpdf, flushend,subfigure}
\ifCLASSINFOpdf
  \usepackage[pdftex]{graphicx}
  \graphicspath{{../pdf/}{../jpeg/}}
  \DeclareGraphicsExtensions{.pdf,.jpeg,.png}
\else
  \usepackage[dvips]{graphicx}
  \graphicspath{{../eps/}}
  \DeclareGraphicsExtensions{}
\fi
\usepackage{epstopdf}
\usepackage{multirow}
\usepackage{float}
\usepackage[cmex10]{amsmath}
\usepackage {amssymb}
\usepackage{graphicx}
\usepackage{array}
\usepackage{mdwmath}
\usepackage{mdwtab}
\usepackage{eqparbox}
\usepackage{nomencl}
\usepackage{cite}
\usepackage{algorithm}
\usepackage{algorithmic}
\usepackage{makeidx}
\usepackage{ifthen}
\usepackage{subfigure}
\usepackage{caption}
\usepackage{amsthm}
\usepackage{pdflscape}

\usepackage{xcolor}

\newcommand{\beq}{\begin{equation}}
\newcommand{\eeq}{\end{equation}}
\newcommand{\beqn}{\begin{eqnarray}}
\newcommand{\eeqn}{\end{eqnarray}}
\newcommand{\beqno}{\begin{eqnarray*}}
\newcommand{\eeqno}{\end{eqnarray*}}
\newcommand{\bma}{\begin{displaymath}}
\newcommand{\ema}{\end{displaymath}}
\newcommand{\bnu}{\begin{enumerate}}
\newcommand{\enu}{\end{enumerate}}
\newcommand{\bce}{\begin{center}}
\newcommand{\ece}{\end{center}}
\newcommand{\btb}{\begin{tabular}}
\newcommand{\etb}{\end{tabular}}

\newtheorem{theorem}{Theorem}[section]

\newtheorem{proposition}[theorem]{Proposition}

\allowdisplaybreaks[2]

\begin{document}

\title{Resilient Edge Service Placement under Demand and Node Failure Uncertainties}

\author{\IEEEauthorblockN{Jiaming Cheng,~\IEEEmembership{Student Member,~IEEE}, Duong Tung Nguyen,~\IEEEmembership{Member,~IEEE}, \\ and Vijay K. Bhargava,~\IEEEmembership{Life~Fellow,~IEEE}\vspace{-0.5em}
}  
\thanks{Jiaming Cheng and Vijay K.  Bhargava are with the Department of Electrical and Computer Engineering, University of British Columbia, Vancouver, BC, Canada.
Email: \textit\{jiaming, vijayb\}@ece.ubc.ca.

Duong Tung Nguyen is with the  Ira A. Fulton Schools of Engineering, Arizona State University, Tempe, AZ, United States. Email: \textit{duongnt@asu.edu}.
This research was supported, in part, by the Natural Sciences and Engineering Research Council of Canada.
(\textit{Corresponding author}: Duong Tung Nguyen).}}

\maketitle

\begin{abstract}
Resiliency plays a critical role in designing future communication networks. How to make edge computing systems resilient against unpredictable failures and fluctuating demand is an important and challenging problem. 
To this end, this paper investigates a resilient service placement and workload allocation problem for a service provider (SP) who can procure resources from numerous edge nodes to serve its users, considering both resource demand and node failure uncertainties. We introduce a novel two-stage adaptive robust model to capture this problem. The service placement and resource procurement decisions are optimized in the first stage, while the workload allocation decision is determined in the second stage after the uncertainty realization. 
By exploiting the special structure of the uncertainty set, we develop an efficient iterative algorithm that can converge to an exact optimal solution within a finite number of iterations. However, the running time of this iterative  algorithm heavily depends on the  uncertainty set. Therefore, we further present an affine decisions rule approximation approach, which exhibits greater insensitivity to the uncertainty set, to tackle the underlying problem. 
Extensive numerical results demonstrate the advantages of the proposed model and approaches, which can  help the SP make proactive decisions to mitigate the impacts of the uncertainties.

\end{abstract}

\printnomenclature

\section{Introduction}
Edge computing (EC) has emerged as a crucial computing paradigm that complements  the cloud to meet the stringent  requirements of modern applications such as augmented/virtual reality (AR/VR),  autonomous driving, and manufacturing automation. By locating cloud resources closer to users, devices, and sensors, EC significantly reduces network traffic, improves user experience, and enables the deployment of various low-latency and high-reliability applications. The new network architecture has an EC layer positioned between the cloud and the end devices, as shown in  Fig.~\ref{fig: edge}. Each edge node (EN) can reside anywhere along the edge-to-cloud continuum and may comprise  one or multiple edge servers  \cite{wshi16}. An EN can also be co-located with a point of aggregation (POA). Generally, user requests and sensor data are aggregated at POAs (e.g., switches/routers, base stations) before being transmitted to the edge network or the cloud for further processing.

\begin{figure}[ht!]
	\centering
		\includegraphics[width=0.42\textwidth,height=0.15\textheight]{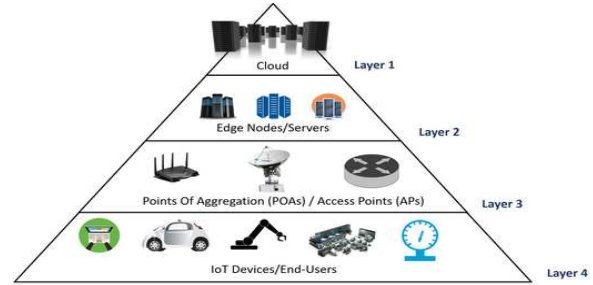}
			\caption{Edge network architecture}
	\label{fig: edge}
\end{figure} 

Despite the tremendous potential, 
EC is still in its nascent stage, and numerous challenges remain to be addressed. One of the major challenges is to establish resilient edge systems that can withstand and recover from disruptions and failures.
Being prepared proactively for an unpredictable future can help protect the network from potential impacts and enable swift recovery from various challenges, ranging from minor misconfigurations to large-scale natural disasters and targeted attacks. The significance of edge resiliency is continuously growing as EC becomes a critical component in the operation of numerous critical infrastructures.

This paper focuses on the resilient service placement and workload allocation problem for a delay-sensitive service, such as AR/VR, real-time translation, remote robotics, and cloud gaming. The service provider (SP) serves a vast number of users situated in different areas. 
To enhance the user experience, the SP can provision the service 
across various ENs to reduce the delay between the users and computing nodes.
The main goal of the SP is to optimize service quality  while minimizing the total operating cost. The service placement decision is typically optimized over a longer time scale than the workload allocation decision. Therefore, a key concern of the SP is \textit{how to make robust placement decisions to maintain good service quality in the presence of unexpected failures and uncertain resource demand}, considering its budget constraint.

The heterogeneity of the ENs poses a significant challenge to selecting suitable nodes for service placement.
Specifically, although placing the service on more ENs can lower the overall delay, it  increases the SP's cost. Furthermore, unlike the traditional cloud with a limited number of large cloud DCs, there are numerous geographically distributed heterogeneous  ENs with different sizes and configurations. In addition, resource prices at different ENs can vary significantly. Therefore, some ENs closer to the users may not be selected due to their higher prices \cite{duongiot}. Along with service placement, the SP must also 
determine the amount of resources to purchase from each selected EN. The service placement and resource procurement decisions are often referred to as service provisioning, 
which is constrained by the operating budget of the SP. 
Given the provisioning decisions, the SP  optimally allocates the workloads from different areas to the selected  ENs with the installed service  to minimize the overall network delay. 
If all system parameters are known exactly, 
the underlying problem can be modeled as a deterministic mixed-integer linear program (MILP) that can be solved using standard MILP solvers. 

Previous studies have predominantly employed deterministic optimization models for optimal edge service placement and workload allocation. However, in practice, numerous system parameters are uncertain, and the SP typically relies on predicted values of these parameters as inputs to the MILP. This deterministic approach can be inefficient if actual values deviate significantly from their forecasts. For instance, if the actual demand is noticeably higher than predicted, under-provisioning may frequently occur, resulting in unmet demand and dropped requests. Conversely, if actual demand is lower than predicted, over-provisioning may lead to unnecessary increases in the SP’s operating costs.

In addition to the demand uncertainty, the unpredictability of node failures can significantly impact the system's performance.
Node failures can result from various factors, including power outages, internal component failures, natural disasters, cyberattacks, and software malfunctions. These unforeseen  events can disrupt EN operations and cause user requests to be reassigned to distant ENs, leading to a suboptimal user experience. Moreover, if the reassigned nodes do not have sufficient resources to handle the demand,
some workload may need to be dropped, which is undesirable.  Thus, ensuring resilience to node failures is crucial in EC systems to maintain good service quality.
The challenge lies in provisioning adequate  edge resources to minimize the total cost while ensuring  the service works properly and can automatically adapt and optimize its operation during extreme events.

To this end, we propose a novel resilience-aware two-stage robust model to compute an optimal resource provisioning solution that can hedge against   resource demand and node failure uncertainties. In decision-making under uncertainty, stochastic optimization (SO) is a popular tool that has been widely applied to various problems in cloud and edge computing \cite{dusit12,dusit17,hbad20}. However, one limitation of SO is its requirement for knowledge of the probability distribution of uncertain data, which can be difficult to obtain. Moreover, the realization of uncertainties during the operational period may not necessarily follow historical patterns \cite{duongiot}, and modeling discrete uncertainties, such as node failures, can pose additional challenges. 

Hence, we advocate robust optimization (RO) \cite{RObook} as a potential solution for addressing uncertainties in the underlying problem.  
To model uncertainties, RO utilizes parametric sets, referred to as uncertainty sets, which are easier to derive than exact probability distributions. 
Additionally,  EN failures can  be intuitively modeled  using cardinality-constrained sets \cite{uncertainty}, without requiring hard-to-obtain probabilistic failure models \cite{chemo20,fhe22}. 
While RO has been utilized in computer networking, 
previous works have mostly focused on single-state static robust models \cite{mjoh15,rkan22,Niyato13,Qu2021,Fair_EC_protect}, which can be overly conservative and fail to capture the sequential nature of certain decisions.

In our  problem, the  provisioning decision must  be made before the actual workload allocation. Thus, instead of using the conventional static RO approach we propose to leverage the two-stage adaptive robust optimization (ARO) approach \cite{CCG,ARO} to tackle the problem. To the best of our knowledge, this is the \textit{first two-stage adaptive robust model for the resilient edge service placement and workload allocation} problem, which jointly considers both demand and node failure uncertainties in EC. Furthermore, both continuous and discrete uncertainties are integrated into a unified RO framework.  By allowing the SP to optimize recourse workload allocation decisions in the second stage, the proposed model is less conservative than the traditional single-stage static RO approach adopted in prior literature \cite{mjoh15,rkan22,Niyato13,Qu2021,Fair_EC_protect}, where  all decisions are optimized at the same time.  The proposed robust solution provides the SP with proactive measures to withstand unexpected failures, ensuring uninterrupted operations while preserving service quality.

It is worth emphasizing that this work differs substantially    from our previous work \cite{duongiot}, which was the first to consider two-stage RO for EC. The differences are in terms of both modeling and solutions. 
In terms of modeling, \cite{duongiot} focuses solely on economic  objectives, whereas  this paper is mainly concerned with ensuring good service quality during failures. Additionally, unlike \cite{duongiot}, we do not consider the cloud option since the service is delay-sensitive.
 Moreover, while \cite{duongiot} only considers  demand uncertainty, this work integrates multiple uncertainties, including both continuous
and discrete uncertain parameters (i.e., resource demand
and node failures), into a unified robust model. Due to
different design goals, the insights gleaned from \cite{duongiot} and
this work are  distinct.

In terms of techniques, we develop two new and efficient solutions for solving the formulated two-stage robust adaptive optimization problem. The solution approach in \cite{duongiot} relies on the column-and-constraint generation (CCG) method and employs the Karush–Kuhn–Tucker (KKT) conditions to solve the second-stage subproblem. However, this approach leads to a large-scale reformulation of the subproblem at each iteration, with a considerable number of complementary constraints, which is difficult to solve. Additionally, the big-M method to tackle the complementary constraints in \cite{duongiot} suffers from weak relaxations. Instead, we propose to use LP duality \cite{lpbook} to reformulate the subproblem, which significantly speeds up the computational time  and provides an exact reformulation, thanks to the unique structure of the uncertainty.
Furthermore, since the CCG-based solution approaches are sensitive to the form and size of the uncertainty sets, we introduce a new approximation scheme to efficiently solve certain large-scale problem instances in a reasonable time. 
This approximation algorithm predefines a mapping rule that maps the second-stage decisions as functions of the first-stage decisions and the revealed uncertainties.
The goal is to optimize the coefficients of these mapping functions. 
Overall, our main contributions can be summarized as follows:

\begin{itemize}
\item \textit{Modeling}: We propose a novel resilience-aware two-stage adaptive robust model for joint optimization of edge service placement, resource procurement, and workload allocation. The uncertainties of demand and EN failures are explicitly integrated into the proposed model, which consists of the service placement and resource procurement in the first stage and the workload allocation decision in the second stage.
   
\item \textit{Solution Approaches:} First, by combining CCG \cite{CCG} with LP duality \cite{lpbook} and several linearization techniques, we develop an efficient algorithm to solve the underlying problem in an iterative master-subproblem framework. Unlike \cite{duongiot} and \cite{CCG} that use the  KKT conditions to solve the subproblem, we leverage the special structure of the uncertainty set and propose to solve the subproblem exactly using LP duality \cite{lpbook}. Second, since the running time of CCG-based algorithms is sensitive to the uncertainty sets, we introduce an affine decision rule (ADR) approximation approach, which is independent of the set sizes and provides a scalable solution. Different from CCG,  ADR is an approximation method and does not work in an iterative manner. Instead, in ADR, the second-stage recourse variables are restricted to be affine functions of the uncertain data. Then, we need to optimize not only the placement and workload allocation decisions but also the coefficients of these affine functions. 

\item \textit{Simulation:} Extensive numerical results demonstrate the superior out-of-sample performance of the proposed ARO model compared to several benchmarks, including the heuristic,  deterministic, and stochastic models. We also illustrate the benefits of considering both the demand and failure uncertainties. Furthermore, our experiments show that ADR can achieve comparable performance with respect to the exact CCG-based methods while requiring significantly less computational time. Finally, sensitivity analyses are conducted to examine the impacts of important system parameters on the optimal Solution. 
\end{itemize}

The rest of the paper is organized as follows. In  Section \ref{system_model}, we present the system model and  problem formulation. The  solution approaches are introduced in Section \ref{solution}, followed by the numerical results in Section \ref{sim}. Section \ref{related_work} discusses related work. Finally, the conclusions are given in  Section \ref{conclusion}.

\section{System Model and Problem Formulation}
\label{system_model}

\subsection{System Model}
In this paper, we investigate the edge resource procurement and management problem for an SP offering a delay-sensitive service, such as AR/VR and cloud gaming. The SP does not own any ENs but instead procures computing resources from an EC market comprising  numerous  geographically distributed heterogeneous ENs. The SP has subscribers located in different areas, where user requests in each area are aggregated at an access point (AP). To reduce network delay and enhance  user experience, the SP can place its services directly on the ENs. The major concerns for the SP are TO determine where to install its service and how much computing resources to procure from each EN. Once the edge resources are procured, the SP must optimally allocate the workload  from different APs to the ENs. The SP is assumed to be  a price-taker with a budget B for  edge resource procurement. 

\begin{figure}[ht!]
	\centering
		\includegraphics[width=0.42\textwidth,height=0.14\textheight]{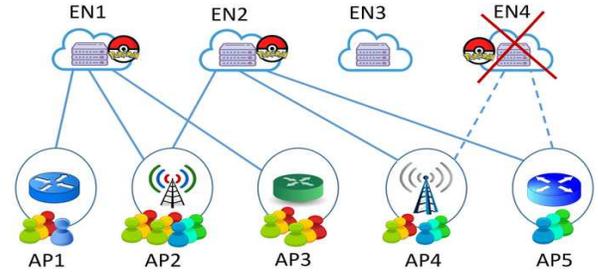}
			\caption{System model}
	\label{fig:model}
\end{figure} 

The system model is shown in Fig.~\ref{fig:model}, which depicts  five areas and four ENs. The service is installed on EN1, EN2, and EN4. In the absence of failures, user requests in areas 4 and 5 can be serviced by EN4, which is the closest EN. However, if  EN4 fails, the workload from these areas can be reassigned to another  EN, such as  EN2, which also has the service installed. Let $i$ be the area index and $\mathcal{I}$ be the set of areas, each of which is represented by an AP. The resource demand in area $i$ is $\lambda_i$. Define $j$ and $\mathcal{J}$ as the EN index and  the set of ENs, respectively.  The resource capacity of EN $j$ is $C_j$. Also, $p_j$ is the price of one unit of computing resource at EN $j$. The delay between AP $i$ and EN $j$ is $d_{i,j}$.
To operate the service, an EN must meet certain hardware and software requirements such as support for Pytorch and Ubuntu. Additionally, for a delay-sensitive service, the requests in an area should be served by ENs that are not too far from that area (e.g., the network delay should not exceed a certain delay threshold).  Hence, we use a binary indicator  $a_{i,j}$ to indicate whether EN $j$ can serve demand from area $i$ or not.  Specifically,  $a_{i,j} = 1$ if the demand from area $i$ can be handled by EN $j$. 

In the first stage, the SP optimizes the service placement and resource procurement decisions.
The placement decision at EN $j$ is represented by a binary variable $t_j$ which equals 1 if the service is installed at the EN. The amount of resource purchased from EN $j$ is denoted by $y_j$. Define $t = (t_1, t_2, \ldots, t_J)$ and $y = (y_1, y_2, \ldots, y_J)$. Given the first-stage decisions, 
in the second stage, the SP optimally allocates the actual demand to different ENs, which have installed the service in the first stage, under the worst-case uncertainty realization. Since the exact demand and EN failures are not known to the SP in the first stage, the procured resources may not be sufficient to serve the actual demand. Hence, a portion of the user requests may be dropped.  Denote by $q_i$ and $P_i$ the amount  of unmet demand and the penalty for each unit of unmet demand in area $i$, respectively. Also, let $x_{i,j}$ be  the workload allocated from area $i$ to EN $j$ in the worst case. 

Depending on specific market settings, the SP may or may not be allowed to adjust the amount of purchased resources during the actual operation period.  In our model, the SP cannot adjust $y_j$ in the second stage, which follows the current practice where cloud providers often require the buyers to purchase resources for usage for a certain minimum amount of time \cite{duongiot,EC2}. Also, during the operation period, the  demand can fluctuate significantly. Thus, the EN owners may not want the SP to frequently adjust $y_j$ for short intervals because the economic benefit is small while the management cost increases. Our model can be easily extended to allow the SP to readjust the amount of resources  in the second stage.

\begin{table}[h!]
\centering
\caption{Notations}
\begin{tabular}{|ll|}
\hline
\multicolumn{1}{|l|}{Notation}            & Definition                                                \\ \hline 
\multicolumn{2}{|c|}{\textbf{Sets and indices}}                                                             \\ \hline
\multicolumn{1}{|l|}{EN, AP}              & Edge Node, Access Point                                \\ \hline
\multicolumn{1}{|l|}{$\mathcal{I}$, I}    & Set and number of areas (APs)                          \\ \hline
\multicolumn{1}{|l|}{$\mathcal{J}$, J}    & Set and number of ENs                                  \\ \hline
\multicolumn{1}{|l|}{$i$, $j$}            & Area (AP) index and EN index                           \\ \hline
\multicolumn{1}{|l|}{$\mathcal{Z}  $}     & Node failure uncertainty set                           \\ \hline
\multicolumn{1}{|l|}{$\mathcal{D}  $}     & Demand uncertainty set                                 \\ \hline
\multicolumn{2}{|c|}{\textbf{System parameters}}                                                              \\ \hline
\multicolumn{1}{|l|}{$C_j$}               & Resource capacity of EN $j$                            \\ \hline
\multicolumn{1}{|l|}{$p_j$}               & Unit price of computing resource at EN  $j$            \\ \hline
\multicolumn{1}{|l|}{$f_j$}               & Service placement cost at EN $j$                       \\ \hline
\multicolumn{1}{|l|}{$l_j^0$}   & $\{0,1\}$, ``1'' if the service is initially available at EN $j$ \\ \hline
\multicolumn{1}{|l|}{$s_j$}               & Storage cost at EN $j$                                 \\ \hline
\multicolumn{1}{|l|}{$a_{i,j}$} & $\{0,1\}$,  ``1'' if EN $j$ can serve demand from area $i$       \\ \hline
\multicolumn{1}{|l|}{$d_{i,j}$}           & Delay between AP $i$ and EN $j$                        \\ \hline
\multicolumn{1}{|l|}{$\beta$}             & Delay penalty parameter                                \\ \hline
\multicolumn{1}{|l|}{$P_i$}               & Penalty parameter for unmet demand in area $i$         \\ \hline
\multicolumn{2}{|c|}{\textbf{Uncertainty-related parameters}}                                                         \\ \hline
\multicolumn{1}{|l|}{$\bar{\lambda}_i$}   & Nominal demand in area $i$                             \\ \hline
\multicolumn{1}{|l|}{$\tilde{\lambda}_i$} & Maximum demand deviation in area $i$                   \\ \hline
\multicolumn{1}{|l|}{$\lambda_i$}         & Actual demand in area $i$                              \\ \hline
\multicolumn{1}{|l|}{$z_j$}               & $\{0,1\}$, ``1'' if EN $j$ fails                       \\ \hline
\multicolumn{1}{|l|}{B}                   & Budget of the SP                                       \\ \hline
\multicolumn{1}{|l|}{K}                   & Node failure uncertainty budget                                    \\ \hline
\multicolumn{1}{|l|}{$\Gamma$}            & Demand uncertainty budget                              \\ \hline
\multicolumn{2}{|c|}{\textbf{Decision variables}}                                                           \\ \hline
\multicolumn{1}{|l|}{$x_{i,j}$}           & Workload from area $i$ to EN $j$    \\ \hline
\multicolumn{1}{|l|}{$y_j$}               & Amount of procured resource at EN $j$               \\ \hline
\multicolumn{1}{|l|}{$q_i$}               & Amount of unmet demand in area $i$                     \\ \hline
\multicolumn{1}{|l|}{$t_j$}               & $\{0,1\}$, ``1'' if the service is installed at EN $j$ \\ \hline
\end{tabular}
\label{notation}
\end{table}
\vspace{-0.2cm}
\subsection{Uncertainty Modeling}
Since the exact demand and node failures usually cannot be predicted accurately at the time of making the first-stage decisions, these  uncertainties need to be properly captured and integrated into the decision-making model of the SP. In RO, uncertain parameters are modeled by uncertainty sets, which represent an infinite number of scenarios. The uncertainty sets are usually constructed to balance between the robustness and conservativeness of the robust solution and to make the underlying optimization problem computationally tractable \cite{RObook}. Similar to our previous work \cite{duongiot}, we employ a polyhedral uncertainty set, which is  widely used in the RO literature \cite{RObook, ADR, ARO, CCG,wang20,Delage21,Pmedium}, to model the demand uncertainty.

Define $\lambda = ( \lambda_1, \lambda_2, \ldots, \lambda_I \big)$  and  $\bar{\lambda} = (\bar{\lambda}_1, \bar{\lambda}_2, \ldots, \bar{\lambda}_I)$ as the vector of  actual demands and  the vector of minimum  demands, respectively, in different areas.
Recall that $I$ is the number of areas (i.e., $I = | \mathcal{I} |$). Let $\tilde{\lambda}_i$ be the maximum deviation of the demand in area $i$ from $\bar{\lambda}_i$. Define $\tilde{\lambda} = (\tilde{\lambda}_1, \tilde{\lambda}_2, \ldots, \tilde{\lambda}_I  )$ as the vector of maximum demand deviations. Then, the actual demand $\lambda_i$ belongs to the range of $[\Bar{\lambda}_i, \Bar{\lambda}_i + \tilde{\lambda}_i]$. The polyhedral demand uncertainty set $\mathcal{D}$ can be expressed as follows \cite{CCG,wang20,Delage21,Pmedium}: 
\begin{align}
    \mathcal{D} =  \bigg\{ \lambda:~ \lambda_i = \Bar{\lambda}_i + g_i \tilde{\lambda}_i; ~ g_i \in [0,1], \forall i;  \sum_i g_i \leq \Gamma \bigg\}, \label{demand_uncertainty}
\end{align}
where  $\Gamma$ is called the demand uncertainty budget, which controls the conservative level of the robust solution \cite{duongiot,RObook,CCG,ADR,ARO}.
The uncertainty budget is set by the SP and can take any value between 0 and $I$. 
The demand uncertainty set enlarges as $\Gamma$ increases. If $\Gamma = I$, $ \mathcal{D}$ becomes a box uncertainty set. 
Without loss of generality, we assume that  $\Gamma$ is an integer.
If $\Gamma$  is a non-integer, the SP can round it up to the closest integer, which indeed slightly enlarges the uncertainty set $\mathcal{D}$ and makes the optimal solution more robust. Thus, restricting $\Gamma$ to be an integer is a mild assumption. 

Besides the demand uncertainty, the SP also faces unpredictable failures of ENs. 
To capture the node failure uncertainty in EC, we employ a cardinality-constrained uncertainty set, which is commonly used to describe discrete uncertainties \cite{uncertainty}. Let $z_j$ be a binary indicator that equals 1 if EN $j$ fails. The failure uncertainty set $\mathcal{Z}$   can be represented as:
\begin{align}
    \mathcal{Z} =  \bigg\{ z_j \in \{0,1\}^{|J|}: \sum_{j} z_{j} \leq K \bigg\}, \label{failure_uncertainty}
\end{align}
where the integer  $K$ expresses the maximum number of ENs that can fail at the same time. In other words, the proposed solution is robust against up to K simultaneous node failures. Obviously, $K \leq |\mathcal{J}| = J$.  As $K$ increases, the optimal solution becomes more robust. However, it is also more conservative (i.e., higher resource provisioning cost). If $K$ is set to zero, it implies that the SP does not consider potential node failures in its decision-making process. The SP can choose suitable values of $\Gamma$ and $K$ to control the level of robustness of the optimal solution. We use the following set $\Xi$ to jointly capture the uncertainties of demand and EN failures.
\begin{align}
\label{uncertainty}
    \Xi = \bigg\{ (\mathbb{\lambda},\mathbf{z}) \in \mathbb{R}_{+}^{|I|} \times \{0,1\}^{|J|}: \mathbb{\lambda} \in \mathcal{D}, \: z \in \mathcal{Z} \bigg\}.
\end{align}

\subsection{Problem Formulation}
We are now ready to describe the two-stage robust model which assists the SP to make resilient provisioning decisions while minimizing the cost and enhancing the user experience. The first-stage decision variables include service placement and  resource procurement decisions. In the second stage, the SP optimizes the actual workload allocation  after the uncertainties are disclosed. The two stages are coupled through the service placement and resource procurement variables.
The SP aims to minimize not only the resource provisioning cost, including the service placement  and edge resource procurement costs, but also the delay and unmet demand  penalty costs.

If the SP decides to place the service onto EN $j$ that does not have the service installed at the beginning of the scheduling period, the service needs to be downloaded from the cloud or a nearby EN, then installed at the EN. In this case, it incurs a cost $f_j$ for installing the service onto EN $j$. Clearly, if the service is available at EN $j$  at the beginning, this cost becomes zero. Let $l_j^0$ be a binary indicator that is equal to 1 if the service is initially available at EN $j$. The total service placement cost $\mathcal{C}^{\sf p}$ over all the ENs can be expressed as follows.
\beqn
    \mathcal{C}^{\sf p} =  \sum_{j \in \mathcal{J}} f_j (1 - l_{j}^{0}) t_j.
\eeqn

If the SP places the service at EN $j$, it needs to pay a storage cost  depending on the size of the service and the length of the scheduling period. Let $s_j$ denote the storage cost  at EN $j$. Then, the total storage cost $\mathcal{C}^{\sf s}$ is:
\beqn
\mathcal{C}^{\sf s} = \sum_{j \in \mathcal{J}} s_j t_j.
\eeqn

The edge resource cost at EN $j$ equals the amount of procured resource $y_j$ multiplied by the resource price $p_{j}$. Thus, the total computing resource cost $C^{\sf c}$ is:
\begin{align}
\mathcal{C}^{\sf c} = \sum_{j \in \mathcal{J}} p_j y_j.
\end{align}

The delay cost between area $i$ and EN $j$ is proportional to the amount of workload allocated from area $i$ to EN $j$ and the network delay between them. Thus, the total delay costs $\mathcal{C}^{\sf d} $ can be given as follows:
\begin{align}
    \mathcal{C}^{\sf d} = \beta \sum_{i \in \mathcal{J}} \sum_{j \in \mathcal{J}} d_{i,j} x_{i,j}.
\end{align}
Finally, the penalty for unmet demand in each area is proportional to the amount of unmet demand and the penalty for each unit of unmet demand in that area. Hence, the total penalty for unmet demand  $ \mathcal{C}^{\sf u}$ over all the areas is:
\begin{align}
    \mathcal{C}^{\sf u} = \sum_{i \in \mathcal{I}} P_i q_i.
\end{align}

Define $h_j = f_j (1 - l_{j}^{0}) + s_j$. The two-stage RO problem of the SP can be written as follows:

\begin{subequations}
\label{ARO_model}
\begin{align}
& \big(\mathcal{P}_1 \big):  \underset{y,~t}{\text{min}} ~\Bigg\{ \sum_{j}  p_{j} y_{j} + \sum_j h_j t_{j} \nonumber \\ 
&  \quad \quad \quad \quad + \underset{(\lambda,z) \in \Xi}{\text{max}} ~ \underset{x,q}{\text{min}} \: \sum_{i} P_{i} q_i + \beta \sum_{i,j}  d_{i,j}  x_{i,j} \Bigg\} \label{SPobj} \\
&\text{s.t.} \: \: \Omega_{1}(y,t) = \bigg\{ \sum_j  p_{j} y_{j} + \sum_j h_j t_{j} \leq B, \label{constr1}\\
& 0 \leq  y_{j} \leq C_j t_j, ~ \forall j; \quad  y_{j} \in \mathbb{Z}, ~ \forall j; t_j \in \{0,1\}, ~\forall j~ \bigg\}\label{constr_var1}\\
&\Omega_2(y,t,\lambda,z) =\bigg\{ \: \sum_{i}  x_{i,j} \leq y_{j} t_j (1 - z_j), \: \forall j \label{constr5}\\
& \sum_{j} x_{i,j} + q_i \geq \lambda_i , \: \forall i \label{constr6}\\
& x_{i,j} \leq a_{i,j} C_{j}, \: \forall i,j \label{constr7}\\
&  x_{i,j} \geq 0, \: \forall i,j; ~~ q_i \geq 0, \forall  i;  ~~ (\lambda, z) \in \Xi \label{constr_uncertainty} \bigg\}.
\end{align}
\end{subequations}
Problem  ($\mathcal{P}_1$) is indeed a trilevel optimization problem. The first level represents the problem of the SP before uncertainties are revealed, which seeks to minimize the SP's cost. The second level,  representing the worst-case realization of the uncertainties $\lambda$ and $z$,  tries to degrade the service quality (i.e., higher delay and more unmet demand) and maximize the SP's cost. The third level represents the optimal workload allocation problem to mitigate the effects of uncertainty realization.

The proposed two-stage robust model can be interpreted as follows. In the first stage, the SP minimizes the provisioning cost, considering the worst-case scenario of the uncertainty realization. The first-stage decisions need to be made before the uncertainties are disclosed. The second stage \textit{max-min} problem expresses the worst-case scenario.The two stages are interdependent through the service provisioning decision in the first stage. The delay penalty $\beta$ and unmet demand penalty $P_i$ reflect the SP's attitude towards the risk of demand fluctuation and node failures in the actual operating stage (i.e., the second stage). Larger values of  $\beta$ and $P_i$ indicate that the SP is more conservative and willing to pay more for the resource provisioning cost in the first stage to mitigate the risk of high delay and dropping requests. 

The set of constraints related to the first stage is captured by $\Omega_1(y,t)$.  Also, $\Omega_2(y,t,\lambda,z)$ expresses all constraints in the second stage. In particular, the budget constraint of the SP is presented in (\ref{constr1}). The provisioning cost cannot exceed the SP's budget. Note that $\beta$ and $P_i$ are used to control the delay and the amount of unmet demand. Intuitively, the provisioning cost tends to increase as $\beta$ and $P_i$ increase. The delay cost and unmet demand penalty cost are virtual costs that the SP  does not have to pay. The first inequalities in (\ref{constr_var1})  enforce that the total amount of resource allocated from each EN $j$ cannot exceed the amount of purchased resource $y_j$ at that node. Furthermore, the SP should purchase resources only from the ENs at which the service is installed (i.e., $t_j$ = 1). The purchased amount is limited by the capacity $C_j$ of the EN. The workload allocation decisions are non-negative, and the placement variables are binary variables, as shown in (\ref{constr_var1}). 

In the second stage, we can only allocate resources from ENs that have installed the service (i.e., $t_j = 1$) and are not in a failure state (i.e., $z_j = 0$). Moreover, the total amount of resource allocated from an EN $j$ cannot exceed the procured amount $y_j$ at that node in the first stage. These aspects are captured precisely by  (\ref{constr5}). Since $y_j$ serves as a parameter in the second stage, it can be observed that there exist bilinear terms $t_j(1-z_j)$ in the inner \textit{max-min} problem. To avoid this bilinear term, we propose to reformulate (\ref{constr5}) using the following equivalent linear constraints:
 \begin{align}
 \label{cap_equivalent}
    \sum_{i} x_{i,j} \leq C_j t_j (1 - z_j),~\forall j; \quad \sum_{i} x_{i,j} \leq y_j, ~\forall j.
\end{align}
Constraints (\ref{constr6}) represent that the resource demand in each area can either be served by some ENs or be dropped (i.e., $q_i$). The demand from area $i$ can be served by an EN $j$ only if $a_{i,j} = 1$, as shown in (\ref{constr7}). Finally, (\ref{constr_uncertainty}) enforces the feasible regions for the uncertainties and second-stage variables. 
To illustrate the impact of  delay requirements on the system performance, we assume that $a_{i,j}$ depends only on the delay $d_{i,j}$ between  area $i$ and EN $j$ and the maximum delay threshold $D^{\text{max}}$. Specifically, we have: 
\vspace{-0.18cm}
\begin{equation}
\label{avaliblity}
   a_{i,j} = \left\{
\begin{array}{ll}
      1, & d_{i,j} \leq D^{\text{max}} \\
      0, & d_{i,j} > D^{\text{max}} \\      
\end{array}, ~~  \forall i,j. \right. 
\vspace{-0.05cm}
\end{equation}
Equation  (\ref{avaliblity}) implies that for a delay-sensitive service, requests from each area should be served by ENs that are not too far from that area. Thus,  an EN $j$ can only serve user requests from area $i$ (i.e., $a_{i,j}= 1$)  if the delay between them is within the threshold $D^{\sf max}$.

\textit{Remark:} It is worth emphasizing that the main goal of our proposed model is to find the optimal service placement and resource procurement decisions, which need to be determined before the uncertainties are disclosed. 
The first-stage decisions $y$ and $t$ are robust against any uncertainty realization. In practice, the proposed system can be implemented as follows. The SP first solves the problem  ($\mathcal{P}_1$) to obtain the optimal values of $y^*$ and $t^*$. After observing the actual realization of the demand $\lambda^{\sf a}$ and failures $z^{\sf a}$, which is not necessarily the worst-case realization, the SP solves the following linear problem to obtain the optimal resource allocation decisions $x$ and $q$ in the actual operation stage:
\begin{subequations}
\label{actual}
\begin{align}
 \underset{x,~q}{\text{min}} \quad  & \sum_{i} P_{i} q_i + \beta \sum_{i,j}  d_{i,j}  x_{i,j} \\
\text{s.t.}: \quad &   \sum_{i}  x_{i,j} \leq y_j^* t_j^* (1 - z_j^{\sf a}), ~ \forall j \label{ac1}\\
& \sum_{j} x_{i,j} + q_i \geq \lambda_i^{\sf a} , \: \forall i \label{ac2}\\
& ~ 0 \leq x_{i,j} \leq a_{i,j} C_{j}, ~ \forall i,j; ~~ q_i \geq 0, \forall  i. \label{ac3}
\end{align}
\end{subequations}

\section{Solution Approaches}
\label{solution}

In this section, we present an exact and iterative solution as well as an affine decision rule (ADR) approximation approach to solving the formulated trilevel problem   ($\mathcal{P}_1$). First, we can show that ($\mathcal{P}_1$) can be written as a large-scale MILP by enumerating over the set of extreme points of the uncertainty set $\Xi$. Specifically, if the innermost minimization problem in ($\mathcal{P}_1$) is feasible, we can solve its dual maximization problem instead. Hence, we write the second-stage bilevel problem as a max-max problem, which is simply a maximization problem over $(\lambda, z)$, and dual variables associated with the constraints of the innermost problem. Because the resulting linear maximization problem is optimized over two disjoint polyhedra,  the optimal solution of the second-stage problem occurs at an extreme point of set $\Xi$. 

Let $\Xi^{\sf e} = \{\xi^1, \xi^2, \ldots, \xi^R \}$ be the set of $R$ extreme points of  $\Xi$, where $\xi^l = (\lambda^l,z^{l})$ is the $l$-th extreme point of  set $\Xi$. Note that $\lambda^l = (\lambda_1^l, \lambda_2^l, \ldots, \lambda_I^l)$ and $z^l = (z_1^l, z_2^l, \ldots, z_J^l)$. Then, problem ($\mathcal{P}_1$) is equivalent to:
\beqn
&& \underset{y,t \in \Omega_{1}(\mathbf{y,t}) }{\text{min}} \Bigg\{   \sum_{j}  p_{j} y_{j} + \sum_j h_j t_{j} \\  \nonumber
&&  + ~ \underset{(\lambda,z) \in \Xi^{\sf e}}{\text{max}} ~ \underset{(x,q) \in \Omega_2(y,t,\lambda,z)}{\text{min}}  \sum_{i} P_{i} q_i + \beta \sum_{i,j} d_{i,j}  x_{i,j} \Bigg\}. 
\eeqn
By enumerating all extreme points of $\Xi^{\sf e}$, it is easy to see that this problem is equivalent to the following MILP:
\begin{subequations}
\label{ARO1}
\begin{align}
\label{objaro}
&   \underset{y,t,\eta}{\text{min}} ~  \sum_j p_{j} y_{j} + \sum_j h_j t_j  + \eta \\
& \text{s.t.} \quad  \eta \geq \sum_i P_i q_i^{l} + \beta \sum_{i,j} d_{i,j}  x_{i,j}^{l}, \: \forall  l \leq R\\
& (x^l, \: q^l) \in \Omega_2(y,t,z^l,\lambda^l),~~ \forall l \leq R.  \label{aroc2}
\end{align}
\end{subequations}
When $\Xi$ contains a large number of extreme points, solving this MILP may be not practically feasible.
Instead of solving (\ref{ARO1}) for all extreme points in $\Xi$, we can try to solve this problem for a subset of  $\Xi^{\sf e}$. Clearly, this relaxed problem contains only a subset of constraints of the minimization problem (\ref{ARO1}). Thus, this relaxed problem gives us a lower bound (LB) for the optimal value of problem  (\ref{ARO1}). By gradually adding more constraints (i.e.,  more extreme points)  to the relaxed problem, we can improve the LB of the original problem ($\mathcal{P}_1$). 
This is indeed the core idea behind the CCG method \cite{CCG} that allows us to   solve ($\mathcal{P}_1$) without explicitly solving the problem (\ref{ARO1}).  

\subsection{ Column-and-Constraint Generation (CCG)}
 To solve  ($\mathcal{P}_1$), we first develop an iterative algorithm based on the  CCG method 
\cite{CCG} that decomposes the original two-stage robust problem into a master problem (\textit{MP}), which is a relaxation of the problem (\ref{ARO1}), and a bilevel max-min subproblem representing the second stage. The optimal value of the \textit{MP}
provides an LB, while the optimal
the solution to the subproblem helps us compute an upper bound (UB) for the optimal value 
of the original problem.  
Also, the optimal solution to the subproblem provides a significant extreme point that is used to update the \textit{MP}. The optimal solution $y$ and $t$ of the \textit{MP} is used to update the subproblem. By iteratively solving an updated \textit{MP} and a modified subproblem, the UB and LB are improved  after every iteration. Thus, CCG is guaranteed to converge to the optimal solution of the original problem in a finite number of iterations.

\subsubsection{Master Problem}
Initially, the \textit{MP} contains no extreme points. A new extreme point is added to the MP at every iteration. Thus, at iteration $r$, the \textit{MP} has $r$ extreme points and can be written as:
\begin{subequations}
\label{masterp}
\begin{align}
& \underset{y,t,\eta}{\text{min}} ~ \sum_{j} p_{j} y_{j} + \sum_j h_j t_j + \eta \label{MPobj} \\ 
& \text{s.t.} \quad (\ref{constr1}) - (\ref{constr_var1}) \nonumber\\
& \eta \geq \sum_i P_i q_i^{l} + \beta \sum_{i,j} d_{i,j} x_{i,j}^{l}, \forall l \leq r  \label{eq:eta}\\
& \sum_i x_{i,j}^l \leq y_j,  ~\forall j \\
& \sum_i x_{i,j}^l \leq C_j t_j (1 - z_j^{l,*}),  ~\forall j,l \leq r \\
& \sum_j x_{i,j}^l + q_i^l \geq \lambda_i^{l,*}, ~\forall  l \leq r\\
&  0 \leq x_{i,j}^{l} \leq a_{i,j} C_{j}, ~\forall i,j,l \leq r; ~~ q_i^l \geq 0, ~\forall i,l \leq r, \label{MPend}
\end{align}
\end{subequations}
where $\{ (\lambda^{1,*}, z^{1,*}), (\lambda^{2,*}, z^{2,*}), \ldots, (\lambda^{r,*}, z^{r,*}) \}$ is the set of optimal solutions to the subproblem in all previous iterations up to iteration $r$. Note that $\lambda^{l,*} = (\lambda_1^{l,*}, \lambda_2^{l,*}, \ldots, \lambda_I^{l,*})$ and $z^{l,*} = (z_1^{l,*}, z_2^{l,*}, \ldots, z_J^{l,*}),~\forall l$.
Clearly, the problem in (\ref{masterp}) is a MILP. The optimal solution to this \textit{MP} includes the optimal service placement ($t_j^{r+1,*}, \forall j$), resource procurement ($y_j^{r+1,*}, \forall j$), second-stage cost ($\eta^{r+1,*}$), as well as $x^{l,*}$ and $q^{l,*},~\forall l \leq r$.  The optimal service placement and resource procurement decisions $y^{*,r+1}$ and $t^{*,r+1}$  serve as input to the subproblem described in Section \ref{subprob}.
The \textit{MP} contains only a subset of constraints of the problem (\ref{ARO1}), which is equivalent to the original problem  ($\mathcal{P}_1$). Hence, the optimal value of the \textit{MP} is a LB for the optimal value of the original problem. The LB achieved after solving the \textit{MP} at iteration $r$ is:
\begin{align}
\label{eq:LBupdate}
    LB = \sum_j p_{j} y_{j}^{r+1,*} + \sum_j h_j t_j^{r+1,*} + \eta^{r+1,*}.
\end{align}

\subsubsection{Subproblem}
\label{subprob}
The subproblem is a bilevel max-min problem  representing the decision-making process of the SP in the second stage. Specifically, given the first-stage decisions $y$ and $t$, the subproblem is given as follows:
\begin{align}
\label{SPP}
\mathcal{Q}(y,t) =  \underset{ (\lambda,z) \in \Xi}{\text{max}} ~\underset{(x, q) \in \Omega_2(y,t,\lambda,z)}{\text{min}}  \sum_i P_i q_i + \beta \sum_{i,j} d_{i,j} x_{i,j}.
\end{align}
The inner  problem can be written explicitly as:
\begin{subequations}
\label{innermostp}
\begin{align}
& \underset{x,~q}{\text{min}}  ~~\sum_i P_i q_i + \beta \sum_{i,j} d_{i,j} x_{i,j} \label{iSPobj} \\  \label{iSPstart}
& \text{s.t.} ~ \sum_{i}  x_{i,j} \leq C_j t_j (1 - z_j),~ \forall j \quad &(u_j^1)\\
& \sum_{i}  x_{i,j} \leq y_j, \: \forall j \quad &(u_j^2) \\
& \sum_{j} x_{i,j} + q_i \geq \lambda_i , \: \forall i  \quad &(s_i) \\
& x_{i,j} \leq a_{i,j} C_{j}, \: \forall i,j  \quad    &(\pi_{i,j})    \\
&  x_{i,j} \geq 0, \: \forall i,j; ~~ q_i \geq 0, \forall  i, & \label{iSPend}
\end{align}
\end{subequations}
where $u_j^1$, $u_j^2$, $s_i$, and $\pi_{i,j}$ are dual variables associated with the corresponding constraints. Also, $t$ and $y$ in problem (\ref{innermostp}) are the optimal placement and procurement to the latest \textit{MP}. Thus, at iteration $r$, for the subproblem, we have  and $y_j = y_j^{r+1,*}$ and $t_j = t_j^{r+1,*} ,~\forall j$. It can be observed that problem (\ref{innermostp}) is feasible for any  uncertainty realization $(\lambda, z)$ as well as any values of $y$ and $t$ because ($x_{i,j} = 0, ~\forall i,j$, $q_i = \lambda_i, ~\forall i$) is always a feasible solution to this problem. Therefore, the second-stage problem satisfies the \textit{relatively complete recourse} condition that is required for CCG \cite{CCG} to work.  The relatively complete recourse condition implies that the second-stage problem is feasible for any given values of the uncertainty realization as well as the service placement and resource procurement computed by the \textit{MP}. 

The subproblem is a bilevel max-min problem, which is difficult to solve. To make it easier to follow the CCG method, we temporarily assume that there is an oracle that can output an optimal solution to the problem (\ref{SPP}) for any given values of $y$ and $t$. 
Let $(\lambda^{r+1,*}, z^{r+1,*}, x^{r+1,*}, q^{r+1,*})$ be an optimal solution to the subproblem at iteration $r$. Then, $z^{r+1,*}$ and $\lambda^{r+1,*}$ are used as input to the \textit{MP} in the next iteration.  Also, the UB for the optimal value of the original problem ($\mathcal{P}_1$) can be updated as follows:
\begin{subequations}
\label{UBu}
\begin{align}
    & UB^{r+1} = \sum_j p_{j} y_{j}^{r+1,*} + \sum_j h_j t_j^{r+1,*} + \sum_i P_i q_i^{r+1,*} \nonumber \\ 
    & + \beta \sum_{i,j} d_{i,j} x_{i,j}^{r+1,*}, \\
    & UB = \min \Big\{ UB, ~UB^{r+1} \Big\}. \label{UB2}
\end{align}
\end{subequations}

\subsubsection{CCG-based Iterative Algorithm}
Based on the description of the \textit{MP} and the subproblem above, we are now ready to present the iterative algorithm for solving 
the problem ($\mathcal{P}_1$) in a master-subproblem framework, as shown in \textbf{Algorithm \ref{CCGalg}}.

\begin{algorithm}[ht!]
\caption{CCG-based Iterative Algorithm}
\label{CCGalg}
\begin{algorithmic}[1]
\STATE Initialization:  set  $r=0$, $LB = -\infty$, and $UB =+\infty$. 
\REPEAT 
  \STATE Solve the  \textit{MP} in (\ref{masterp}) to obtain an optimal solution $(y^{r+1,*}, t^{r+1,*}, \eta^{r+1,*})$   and update LB according to (\ref{eq:LBupdate}).
  \STATE Solve the subproblem (\ref{SPP}) with $y = y^{r+1,*}$ and $t = t^{r+1,*}$ to obtain an extreme point ($\lambda^{r+1,*}$, $z^{r+1,*}$) and update UB following (\ref{UBu}).
  \STATE Update $\lambda^{r+1} = \lambda^{r+1,*}$ and $z^{r+1} = z^{r+1,*}$, which are used to create new cuts in the \textit{MP} in the next iteration. Also, update $r = r+1$. Go to Step 3.
\UNTIL {$\frac{UB -LB}{UB} \leq \epsilon$ }\\
\STATE 
Output: optimal placement and resource procurement decisions $(y^*, t^*)$.
\end{algorithmic}
\end{algorithm}

The CCG algorithm starts by solving  an \textit{MP} in Step 3  to find an optimal placement and procurement solution, which will serve as input to the subproblem in Step 4. The subproblem gives us a new extreme point, which represents the worst-case uncertainty scenario for the given optimal $y$ and $t$ in Step 3. This extreme point is used to generate new cuts (i.e., new constraints) to add to the \textit{MP} in the next iteration. Since new constraints related to the uncertainties $z$ and $\lambda$ are added to the \textit{MP} at every iteration, the feasible region of the \textit{MP} decreases. Hence, the LB is weakly increasing (i.e., improved) after each iteration. By definition in  (\ref{UBu}), the UB is non-increasing. Since $\Xi^{\sf e}$ is a finite set with $R$ elements while  the subproblem produces a new extreme point at every iteration, \textbf{Algorithm \ref{CCGalg}} will converge in a finite number of iterations. 

\begin{proposition}
\label{prop1}
\textit{\textbf{Algorithm \ref{CCGalg}} } \textit{converges to the optimal solution to the original problem ($\mathcal{P}_1$) in $O(R)$ iterations.}
\end{proposition}
\begin{proof}
Please refer to \textit{Appendix \ref{proofcon}} \cite{arxiv22}.
\end{proof}

Overall, \textbf{Algorithm \ref{CCGalg}} converges to the optimal solution within a finite number of iterations depending on the number of extreme points  $R$ of the uncertainty set \textbf{$\Xi^{\sf e}$}. 

\subsubsection{Duality-based Reformulation for the Subproblem} 
\textbf{Algorithm \ref{CCGalg}} requires solving a bilevel max-min subproblem at every iteration. In the previous section, we assume that there exists an oracle for solving the subproblem (\ref{SPP}). In the following, we present an efficient approach to implementing this oracle. In the original CCG paper \cite{CCG}, the authors utilize the  KKT conditions to transform the bilevel subproblem into an equivalent MILP. 
The KKT-based reformulation of the subproblem is presented in \textit{Appendix} \ref{kkta} \cite{arxiv22}. 
Instead of the KKT-based reformulation as in \cite{CCG},  we propose an alternative approach that converts the   subproblem to a MILP by using LP duality. Compared to the KKT-based reformulation, the duality-based reformulation has the advantage of generating fewer variables and constraints, which reduces computational time, especially for large-scale systems. 

By using LP duality \cite{lpbook}, we first write the dual problem of 
(\ref{innermostp}) 
and then, the subproblem (\ref{SPP}) becomes a max-max problem (i.e., simply a maximization problem). Hence,  the subproblem (\ref{SPP}) is equivalent to:
\begin{subequations}
\label{dualsub}
\begin{align}
\label{NL-optsub}
& \: \underset{u^1,u^2,s,\pi,g,z}{\text{max}} \: \sum_{i} \Bar{\lambda}_{i} s_i + \sum_{i} \tilde{\lambda}_i s_i g_i - \sum_j C_j t_j (1 - z_j) u^{1}_j   \nonumber \\
& - \sum_j y_j u^2_j - \sum_{i,j} a_{i,j} C_{i,j} \pi_{i,j} \\
& \text{s.t.} \quad  s_i \leq P_i , \quad \forall i \\
&  s_i - u^1_j - u^2_j - \pi_{i,j}  \leq \beta d_{i,j}, \quad \forall i,j\\
&  u^{\sf 1}_{j},~  u^{\sf 2}_{j}\geq 0, \: \forall j; \quad s_{i} \geq 0, \quad \forall i \\
& \sum_j z_j \leq K;~~ z_j  \in \{0,1\}, \quad \forall j\\
& \sum_{j} g_i \leq \Gamma; \quad0 \leq g_i  \leq 1, \quad \forall i. \end{align}
\end{subequations}
Note that $\lambda = \Bar{\lambda}_{i} + \tilde{\lambda}_i g_i$ from (\ref{demand_uncertainty}). Also, $y$ and $t$ are parameters in problem ($\ref{dualsub}$). Due to the bilinear terms $s_i g_i$ and $z_j u_j^1$, ($\ref{dualsub}$) is a non-linear optimization problem. Since the term $z_j u_j^1$ is a product of a binary variable and a continuous variable, we can linearize it as follows. Define $U_j = z_j u_j^1$. Then, the bilinear term $z_j u_j^1$  can be replaced by $U_j$ and the following linear equations: $ U_{j} \leq u^1_j, U_{j} \leq M z_{j}, 0 \leq U_{j} \geq u^1_j - M (1 - z_{j}), ~ \forall j$, where $M$ is a sufficiently large number \cite{BigM}.

Although $g_i$'s are continuous,  it is easy to see that (\ref{dualsub}) always has an optimal solution  where $g_i \in \{0,1\},~\forall i$, because $\Gamma$ is an integer. The bilinear term $s_i g_i$ becomes a product of a continuous variable and a binary variable. Therefore, we can linearize  $s_i g_i$ using a set of equivalent linear equations as we did for $z_j u_j^1$. Let $v_{i} = s_{i} g_i,~\forall i$. Then, instead of solving the non-linear problem ($\ref{dualsub}$), we can solve the following MILP:
\begin{subequations}
\label{dualsubmilp}
\vspace{-0.2cm}
\begin{align}
\label{opt_sub}   
& \underset{u^1,u^2,s,\pi,g,z}{\text{max}}  ~\sum_{i} \Bar{\lambda}_i s_i + \sum_{i} \tilde{\lambda}_i v_{i} + \sum_j C_j t_j U_j  \nonumber\\
& - \sum_j C_j t_j u^1_j - \sum_{j} y_j u^2_j - \sum_{i,j} a_{i,j} C_{j} \pi_{i,j} \\
& \text{s.t.} \quad  s_i \leq P_i, ~\forall i;~~ s_i - u_j -\pi_{i,j} \leq \beta d_{i,j}, ~ \forall i,j\\ 
& v_{i} \leq s_{i}, ~v_{i} \leq M g_{i},~ v_{i} \geq s_{i} - M (1 - g_i), ~ \forall i\\
& U_{j} \leq u^1_j, ~~ \forall j;~~ U_{j} \leq M z_{j}, ~~ \forall j\\
& 0 \leq U_{j} \geq u^1_j - M (1 - z_{j}), ~ \forall j \\
& u^1_{j}, u^2_j \geq 0,~ \forall j; ~~ v_i, s_{i} \geq 0. \: \forall i; ~~  \pi_{i,j} \geq 0, \: \forall i,j \label{linear2}\\
& \sum_j z_j \leq K; \: z_j \in \{0,1\}, ~\forall j \\
& \sum_i g_i \leq \Gamma; \: g_i \in \{0,1\}, \: \forall i.
\end{align}
\end{subequations}

When we use the duality-based transformation to solve the subproblem (\ref{SPP}), instead of (\ref{UBu}), we can update the UB as: 
\begin{subequations}
\label{UBuu}
\begin{align}
& UB^{r+1} =   \sum_j p_{j} y_{j}^{r+1,*} + \sum_j h_j t_j^{r+1,*} + \sum_{i} \Bar{\lambda}_i s_i^* + \sum_{i} \tilde{\lambda}_i v_{i}^*  \nonumber \\  
&  + \sum_j C_j t_j U_j^* - \sum_j C_j t_j u^{1,*}_j - \sum_{j} y_j u^{2,*}_j - \sum_{i,j} a_{i,j} C_{j} \pi_{i,j}^* \label{UB1u}\\
& UB = \min \Big\{ UB, ~UB^{r+1} \Big\}, \label{UB2u} 
\end{align}
\end{subequations}
where $(s^*, v^*, U^*, u^{\sf 1,*}, u^{\sf 2, *}, \pi^*)$ is an optimal solution to  (\ref{dualsubmilp}).

\vspace{-0.1in}
\subsection{Affine Decision Rule (ADR) Approach}
\label{ADRsec}
Although \textbf{Algorithm \ref{CCGalg}}   can give an optimal solution to  ($\mathcal{P}_1$)  within a finite number of iterations, its computational time depends on the uncertainty set. From \textit{Proposition \ref{prop1}}, the computational time is sensitive to the number of extreme points of the uncertainty set $\Xi^{\sf e}$, i.e., the number of elements of $\Xi^{\sf e}$. In the worst case, it can take a long time to converge if the uncertainty set has a huge number of extreme points.
Hence, we propose an affine decision rule (ADR) approximation method \cite{ARO, ADR, ADR11}, which is insensitive to the set size, to solve large-scale problem instances. The main idea behind ADR is to restrict the second-stage variables 
to be affine functions of the uncertain data. If these functions are given, we can simply use them to compute the suboptimal recourse decisions for any given realization of the uncertainties. Thus, the goal of  ADR  is to find reasonable functions to approximate the optimal solution. 
 While ADR only provides a suboptimal solution, Kuhn \textit{et al.} showed that ADR performs surprisingly well on many 
 problems and it is even optimal in certain problem classes \cite{ADR11}. This motivates us to examine the performance of ADR for our two-stage robust 
problem  ($\mathcal{P}_1$). 

Let $q_i(\lambda, z)$ and $x_{i,j}(\lambda, z)$ express $q_i$ and $x_{i,j}$ as functions of the uncertainties $\lambda$ and $z$. Then, the original problem ($\mathcal{P}_1$) can be written as follows:
\begin{subequations}
\label{ARCp}
\begin{align}
& (\mathcal{P}_1'): \:  \underset{y,t}{\text{min}} \sum_{j}  p_{j} y_{j} + \sum_j h_j t_{j} \nonumber \\ 
& + \underset{(\lambda,z) \in \Xi}{\text{max}}  \underset{x,q}{\text{min}} \: \sum_{i \in I} P_{i} q_i(\lambda,z) + \beta \sum_{i,j} d_{i,j}  x_{i,j}(\lambda,z)  \\
&\text{s.t.} \: \: y, t \in \Omega_{1}(y,t) \label{constrADR1}\\
& \sum_{i}  x_{i,j}(\lambda,z) \leq C_{j} t_j (1 - z_j), ~  \forall j, \:\forall (\lambda,z) \in \Xi \label{constrADR2}\\
& \sum_{i}  x_{i,j}(\lambda,z) \leq y_{j}, \: \forall j, \: (\lambda,z) \in \Xi \label{constrADR3}\\
& \sum_{j} x_{i,j}(\lambda,z) + q_i(\lambda,z)  \geq \lambda_i, \: \forall i,\: (\lambda,z) \in \Xi \label{constrADR4}\\
& 0 \leq x_{i,j}(\lambda,z) \leq a_{i,j}C_j, \: \forall i,j, \: (\lambda,z) \in \Xi \label{constrADR5}\\
& q_i(\lambda,z) \geq 0, \: \forall i; \: (\lambda,z) \in \Xi. \label{constrADR6}
\end{align}
\end{subequations}

In  ADR, the second-stage variables $x_{i,j}$ and $q_i$ are defined as affine functions of the uncertainties $\lambda$ and $z$. Thus:
\begin{align}
&    x_{i,j}(\lambda,z) = \sum_{e \in \mathcal{I}} A_{i,j}^{e} \lambda_e + \sum_{l \in \mathcal{J}} B_{i,j}^l z_{l} + D_{i,j}, \: \forall i, j  \label{ADR1} \\
 &   q_{i}(\lambda,z) = \sum_{e \in \mathcal{I}} E_{i}^{e} \lambda_e + \sum_{l \in \mathcal{J}} F_{i}^l z_{l} + G_i, ~\forall i, \label{ADR2} 
\end{align}
where $e$ is area index and $l$ is EN index. Also, $A_{i,j}^{e}$, $B_{i,j}^l$, $D_{i,j}$, $E_{i}^{e}$, $F_{i}^l$ and $G_i \in \mathbb{R}$. It can be observed that, for each realization of the uncertain data $\lambda$ and $z$, we can readily compute the value of $w$ and $q$ by using (\ref{ADR1}) and (\ref{ADR2}). Thus, the objective of the ADR approach is to optimize the coefficients $A_{i,j}^{e}$, $B_{i,j}^l$, $D_{i,j}$, $E_{i}^{e}$, $F_{i}^l$ and $G_i$ in (\ref{ADR1}) and (\ref{ADR2}).

By using the ADR (\ref{ADR1}) and (\ref{ADR2}) and the epigraph form for  $(\mathcal{P}_1')$, we obtain the following ADR model (\ref{ADRp2}) for $(\mathcal{P}_1)$:
\begin{subequations}
\label{ADRp2}
\begin{align}
& (\mathcal{P}_1^{\textbf{ADR}}): \underset{\begin{subarray}{c} y,t,\\A,B,D,E,F,G\end{subarray}}{\text{min}} \sum_{j}  p_{j} y_{j} + \sum_{j} h_j t_{j} + \phi \\
&\text{s.t.} \: \: y, t \in \Omega_{1}(y,t) \label{constr_ARC}\\ \label{constr_ADR1}
& \phi \geq  \sum_{i} P_{i} \Big(\sum_{e} E_{i}^{e} \lambda_e + \sum_{l} F_{i}^l z_{l} + G_i \Big) + \beta \sum_{i,j} d_{i,j}  \nonumber  \\ 
& \Big(  \sum_{e} A_{i,j}^{e} \lambda_e + \sum_{l} B_{i,j}^l z_{l} + D_{i,j} \Big), ~ (\lambda,z) \in \Xi \\
& \sum_{i} \Big( \sum_{e} A_{i,j}^{e} \lambda_e + \sum_{l} B_{i,j}^l z_{l} + D_{i,j} \Big) \leq C_j t_j (1 - z_j), \nonumber \\
& ~ \forall j, ~ (\lambda,z) \in \Xi  \label{constr_ADR2}\\
& \sum_{i} \Big( \sum_{e} A_{i,j}^{e} \lambda_e + \sum_{l} B_{i,j}^l z_{l} + D_{i,j} \Big) \leq y_j, ~ \forall j, ~  (\lambda,z) \in \Xi  \label{constr_ADR3}\\
& \sum_{j} \Big(\sum_{e} A_{i,j}^{e} \lambda_e + \sum_{l} B_{i,j}^l z_{l} + D_{i,j} \Big) + \Big(\sum_{e} E_{i}^{e} \lambda_e  \nonumber \\
& + \sum_{l} F_{i}^l z_{l} + G_i \Big) \geq \lambda_i, \: \forall i,~  (\lambda,z) \in \Xi   \label{constr_ADR4}\\
& \sum_{e} E_{i}^{e} \lambda_e + \sum_{l} F_{i}^l z_{l} + G_i \geq 0, ~ \forall i, ~(\lambda,z) \in \Xi  \label{constr_ADR5} \\
& \sum_{e} A_{i,j}^{e} \lambda_e + \sum_{l} B_{i,j}^l z_{l} + D_{i,j} \geq 0, \forall i,j,  (\lambda,z) \in \Xi  \label{constr_ADR6}\\ 
& \sum_{e} A_{i,j}^{e} \lambda_e + \sum_{l} B_{i,j}^l z_{l} + D_{i,j} \leq a_{i,j} C_j, \forall i,j, ~ (\lambda,z) \in \Xi. \label{constr_ADR7}
\end{align}
\end{subequations}
To solve the ADR model (\ref{ADRp2}), we first need to convert each robust constraint into a solvable form. Specifically, we employ LP duality \cite{lpbook} to reformulate each robust constraint into an equivalent set of linear equations. Due to space limitation, 
please refer to \textit{Appendix \ref{RC_reformulation}} \cite{arxiv22} for more details. After converting each robust constraint into a set of linear equations, we obtain the following MILP:
\begin{align}
& \underset{\begin{subarray}{c} y,t\\A,B,D, E,F,G\end{subarray}}{\text{min}} \sum_{j}  p_{j} y_{j} + \sum_{j} h_j t_{j} +　\phi \nonumber \\
&\text{s.t.} ~~ y ,t \in \Omega_{1}(y,t) \label{MILP_ADR}\\ 
& \quad \quad (\ref{ADRconstraint1}) - (\ref{ADRconstraint7})  ~\text{in} ~\textit{Appendix \ref{RC_reformulation} \cite{arxiv22}}.
\end{align}

Although the ADR approach only gives us a suboptimal solution, our numerical results show that the ADR solution is quite close to the exact optimal solution. Thus, it provides a good approximation scheme for our problem. Furthermore, for some special cases, we can show that the ADR approach generates an exact optimal first-stage decision for the original problem $(\mathcal{P}_1)$. Note that we  only care about the first-stage decisions because in the operation stage, the SP will eventually reoptimize the second-stage decisions by solving LP problem (\ref{actual}) after the uncertainties are revealed.

We have the following proposition in regard to the optimality of the ADR approach.
\begin{proposition}[Optimality of ADR \cite{ARO}] \label{proof_opt_ADR}
    If the uncertainty set $\Xi$ is a simplex (i.e., the convex hull of its vertices), the affine policy can give an optimal first-stage solution. 
\end{proposition}

From \textit{Proposition \ref{proof_opt_ADR}}, we can show that there exist some special cases where the ADR approach gives an optimal solution to the original problem $(\mathcal{P}_1)$. For example, consider the case where $K = 0$ and $\Gamma = 1$. In other words, the SP considers only demand uncertainty, and set $\Xi$ is the same as set $\mathcal{D}$. Also, $\sum_{i \in \mathcal{I}} g_i \leq 1$. Let $g_0 \geq 0$ be a slack variable such that $g_0 + \sum_{i \in \mathcal{I}} g_i = 1$. We consider the following $I + 1$ vertices of $\Xi$ (i.e., $\mathcal{D}$): $(\bar{\lambda}_1,\bar{\lambda}_2,\dots,\bar{\lambda}_I )$, $(\bar{\lambda}_1 + \tilde{\lambda}_1,\bar{\lambda}_2,\dots,\bar{\lambda}_I )$, $(\bar{\lambda}_1 ,\bar{\lambda}_2 + \tilde{\lambda}_2,\dots,\bar{\lambda}_I )$, $\dots$, $(\bar{\lambda}_1 ,\bar{\lambda}_2,\dots,\bar{\lambda}_I + \tilde{\lambda}_I)$. Recall from (\ref{demand_uncertainty}), we have $\lambda_i = \bar{\lambda}_i + g_i \tilde{\lambda}_i,~\forall i$. Since $g_0 + \sum_{i \in \mathcal{I}} g_i = 1$, it is easy to see that:
\beqn
 \lambda &=& (\bar{\lambda}_1 + g_1 \tilde{\lambda}_1, \bar{\lambda}_2 + g_2 \tilde{\lambda}_2, \ldots, \bar{\lambda}_I + g_I \tilde{\lambda}_I)   \\ \nonumber
&=& g_0 (\bar{\lambda}_1,\bar{\lambda}_2,\dots,\bar{\lambda}_I ) +  \sum_{i \in \mathcal{I}} g_i (\bar{\lambda}_1, \ldots ,\bar{\lambda}_i + \tilde{\lambda}_i,\ldots,\bar{\lambda}_I ).
\eeqn
Hence, the uncertainty set $\Xi$ is a simplex in this case. As a result, from \textit{Proposition \ref{proof_opt_ADR}}, ADR gives an optimal solution to problem  $(\mathcal{P}_1)$  when $K = 0$ and $\Gamma = 1$. By following a similar procedure, we can also show that ADR gives an optimal solution to problem  $(\mathcal{P}_1)$  when $K = 1$ and $\Gamma = 0$.

It can be seen from  (\ref{ADRp2}),  ADR is not sensitive to the characteristics of the uncertainty set since the robust constraints are dualized. However, the network size  significantly affects the  ADR reformulation's complexity. Table \ref{table:SizeComparision} presents  the number of constraints and variables in the ADR reformulation, which depends on the network size. 

\begin{table}[h]
\centering
\begin{tabular}{|c|c|}
\hline
\# constraints      & $IJ(4I+4J+11) + 4I(I + 1) +3J(J + 4) + 5$ \\ \hline
\# variables & $IJ(2I+2J+13) + I(3I+J+3) + J(2J+7)$ \\ \hline
\end{tabular}
\caption{The size of the ADR reformulation.}
\label{table:SizeComparision}
\vspace{-0.2cm}
\end{table}

\section{Numerical Results}
\label{sim}
\subsection{Simulation Setting}
\label{setting}
We consider an edge system comprising 20 areas, with each area having an EN (i.e., $I = |\mathcal{I}|  = |\mathcal{J}| = J = 20$). We will also perform sensitivity analysis on larger networks with more than 20 areas. Since we are not aware of any public data for edge networks, similar to \cite{Jia17, Jia18, duongiot, duongtcc}, we adopt the popular Barabasi-Albert (BA) model to generate  a random scale-free edge network topology with 100 nodes. We extract a subset of nodes out of these 100 nodes to generate the edge network topology. The link delay between two adjacent nodes on the BA network is  generated randomly within the range of $[2,~10]$ \textit{ms} \cite{duongiot}. Then, the network delay between an AP $i$ and EN $j$ is the delay of the shortest path between them.  By employing Dijkstra's shortest path algorithm, we can calculate the network delay between air pair of nodes. In the \textit{default setting}, we assume that all ENs are eligible to serve demand from every area, i.e., $a_{i,j}$ is set to be 1, $\forall i,j$. In Fig.~\ref{fig:delay_threshold}, we will vary the values of $D^{\sf max}$ and $a_{i,j}$ when studying the impact of the delay requirements on the system performance.

Using the hourly price of the \textit{m5d.xlarge} Amazon EC2 instance \cite{EC2} as a reference, the unit resource prices at the ENs are randomly generated from $\$0.02$ to $\$0.06$ per vCPU-hour. Also, the capacities ($C_j, \forall j$) of the ENs are set randomly among $32$, $48$, and $64$ vCPUs. The service placement and storage costs ($h_j$) are randomly generated  between $\$0.1$ and $\$0.2$. The budget ($B$) of the SP is set to be $20$. We also assume that the service is not available on any EN at the beginning (i.e., $l_j^0 = 0, \forall j)$. By analyzing the  trace in \cite{TUDelft},  we  randomly generate the nominal demand in each area between 5 and 40. Also, 
define $\alpha$ as the ratio between maximum demand deviation $\tilde{\lambda}_i$ and the nominal demand $\bar{\lambda}_i$ (i.e., $\tilde{\lambda}_i = \alpha \bar{\lambda}_i, \forall i$). 
In the default setting, $\Gamma = 5$,  $K = 2$,  $\beta = 0.1$, $\alpha = 0.6$, $B = 20$, and $P_i = P = 0.5, \forall i$. We also vary these parameters during sensitivity analyses. The main simulation data is summarized in Table \ref{table:parameter_setting}. All the experiments are conducted in
MATLAB using CVX (http://cvxr.com/cvx/) 
and Gurobi (https://www.gurobi.com/) 
on a  laptop with an Intel Core i7-11700KF CPU and 16GB of RAM. 

\begin{table}[H]
\centering
\begin{tabular}{|l|l|}
\hline
\textbf{Parameters} & \textbf{Values}          \\ \hline
Network size (AP, EN): $(I,J)$ & $(20,20)$          \\ \hline
Unit resource price  & $[0.02,0.06]$ (\$ per hour)\\ \hline
Resource capacity ($C_j, \forall j$)  & $\{32,48,64\}$   vCPU     \\ \hline
Service placement and storage cost ($h_j, \forall j$) & $[0.1, 0.2]$ (\$)\\ \hline
Delay penalty ($\beta$)  & $0.1$ (\$ per  ms)\\ \hline
Budget ($B$) & $20$ (\$) \\ \hline
Unmet demand penalty ($P_i, \forall i$) & $0.5$ (\$ per vCPU) \\ \hline
Uncertain budget ($\Gamma, K$) & $\Gamma = 5$, $K = 2$ \\ \hline 
\end{tabular}
\caption{Simulation data}\label{table:parameter_setting}
\vspace{-0.2cm}
\end{table}

\subsection{Performance Evaluation}
\label{operationCompare}
First, we compare the performance of the proposed two-stage robust (\textbf{ARO)} model $(\mathcal{P}_1)$ with three benchmarks:
\begin{enumerate}
    \item \textit{Deterministic model} (\textbf{DET}): see \textit{Appendix \ref{DetermiA}} \cite{arxiv22}. In this model, the SP uses the forecast demand and does not consider EN failures when making decisions.\label{benchmark2}

    \item \textit{Two-stage stochastic model} (\textbf{SO}): see \textit{Appendix \ref{SOappen}} \cite{arxiv22}. This model aims to optimize the expected system performance over a  set of scenarios generated from historical data or a certain probability distribution.

    \item \textit{Heuristic} (\textbf{HEU}): this scheme uses the forecast demand to make the service placement and workload allocation decision. The SP first sorts the areas based on their demands. Then, the workload in each area is allocated one by one, starting from the highest demand area. For each area, the service is placed onto its closest EN and the whole demand of the area is assigned to this EN if this EN has sufficient capacity. Otherwise, a portion of the demand is assigned to this EN and the remaining is allocated to the second-closest EN. This process is repeated until the demand is fully allocated. The usable capacity of each EN is updated during the process. \label{benchmark3}
\end{enumerate}

For the SO scheme, we generate the demand scenarios from
a truncated multivariate normal distribution \cite{duongiot}. These (training) scenarios are used as input to the SO model shown in  \textit{Appendix \ref{SOappen}} \cite{arxiv22}. The output of the SO model is the optimal service placement and resource procurement solution (t, y). Similarly, we solve the deterministic model in \textit{Appendix \ref{DetermiA}} \cite{arxiv22}, the ARO model ($\mathcal{P}_1$), and run the heuristic scheme to find different optimal service placement and resource procurement solutions. These solutions are computed before the actual operation stage. 
For all four schemes, given the service placement and resource procurement decisions in the first stage, after observing the actual realization of the demand and the EN failures, the SP will solve the linear problem (\ref{actual}) to determine the actual workload allocation and unmet demand decisions. The \textit{actual total cost} of each scheme is the total of the provisioning cost in the first stage and the actual cost in the operation stage (i.e., second stage).

To compare the performance of these schemes, we use Monte-Carlo simulation and generate 1000 (testing) scenarios to represent the demand and failures in the actual operation stage. Specifically, the actual demand is generated from a (truncated) log-normal distribution \cite{duongiot}, and the EN failures are generated randomly from the failure uncertainty set $\mathcal{Z}$. We have also used other distributions to generate the demand and observed similar trends and insights for all the figures. In the following, the \textbf{average cost} 
is the expected actual cost from the $1000$ generated scenarios, whereas the \textbf{worst-case cost} is the highest cost among these scenarios.

Figs. \ref{fig:avg_table}-\ref{fig:worst_table} compare the performance of the four schemes as the failure budget $K$ varies. We can see that the actual costs of all the schemes increase when the number of possible failures $K$ increases. Also, since both the demand and failure uncertainties are explicitly captured in ARO, it significantly outperforms the other schemes. An ARO solution typically installs the service on more ENs and procures resources more evenly among the selected ENs, which makes it more resilient to unexpected failures of ENs and demand fluctuation. When a failure happens at an EN, the workload initially assigned to it can be reallocated to the other nodes. 

The SO scheme performs better than the deterministic and heuristic schemes since the SO model takes the uncertainties into account when making the first-stage decisions. However, SO does not consider worst-case uncertainty realization. Moreover, the actual uncertainty realization may not always follow the historical pattern. Hence, SO performs worse than ARO. Define $\Psi$ as the scaling factor for the unmet demand penalty parameter $P$ compared to its value in the default setting. Figs.~\ref{fig: avg_Penalty}-\ref{fig: Worst_Penalty} further confirm the superior performance of the proposed ARO scheme compared to the other schemes. It is also easy to see that the total cost increases as the unmet demand penalty $P$ increases. Figs.~\ref{fig:avg_table}-\ref{fig: Worst_Penalty} show another advantage of ARO is that its cost does not vary significantly when $K$ or the unmet demand penalty changes. Thus, ARO is a preferred method for SPs who require high-quality of service.

\begin{figure}[ht!]
		\subfigure[Average cost]{
		  \includegraphics[width=0.245\textwidth,height=0.10\textheight]{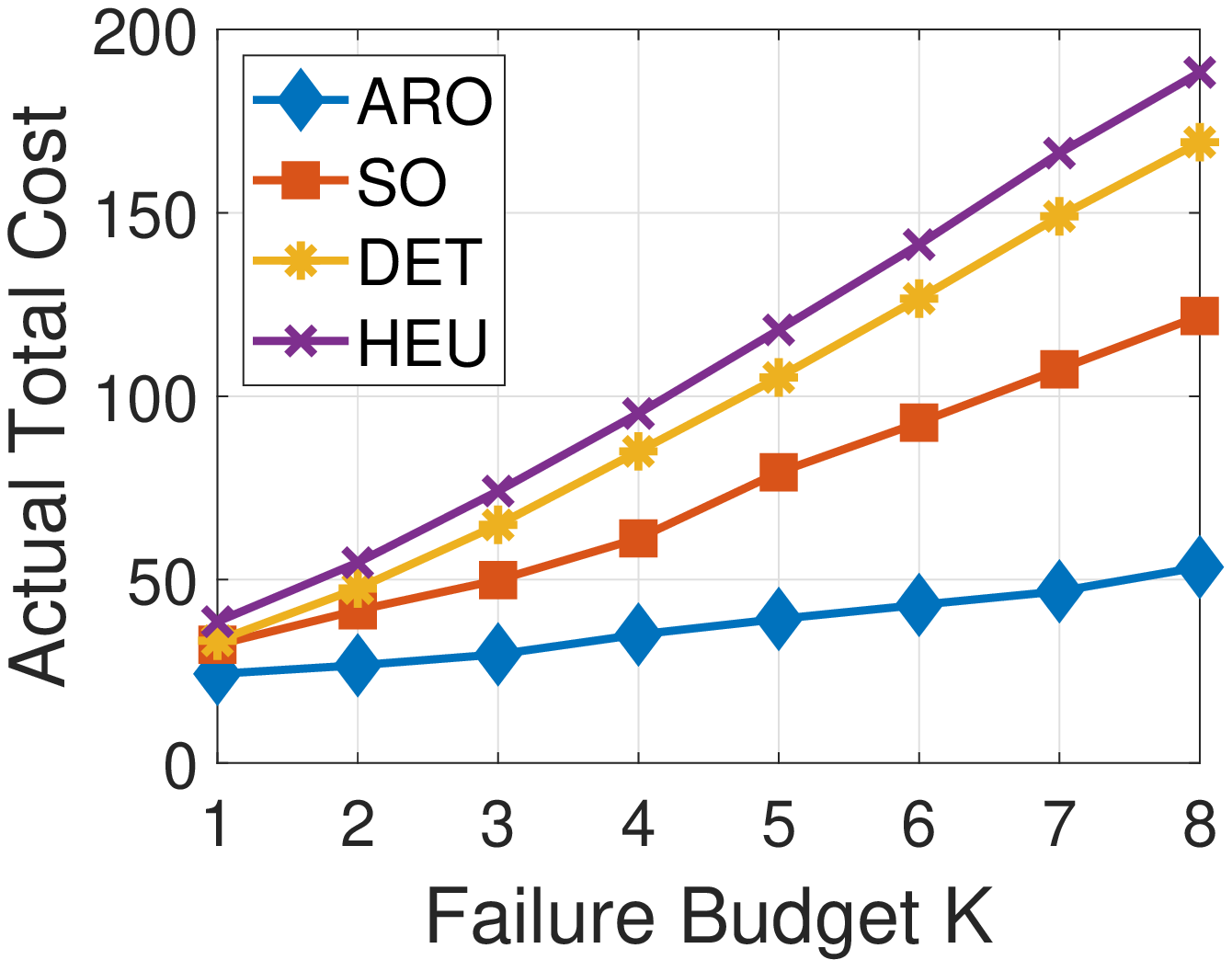}
	    \label{fig:avg_table}
	}   \hspace*{-2.1em} 
		 \subfigure[Worst-case cost]{
	     \includegraphics[width=0.245\textwidth,height=0.10\textheight]{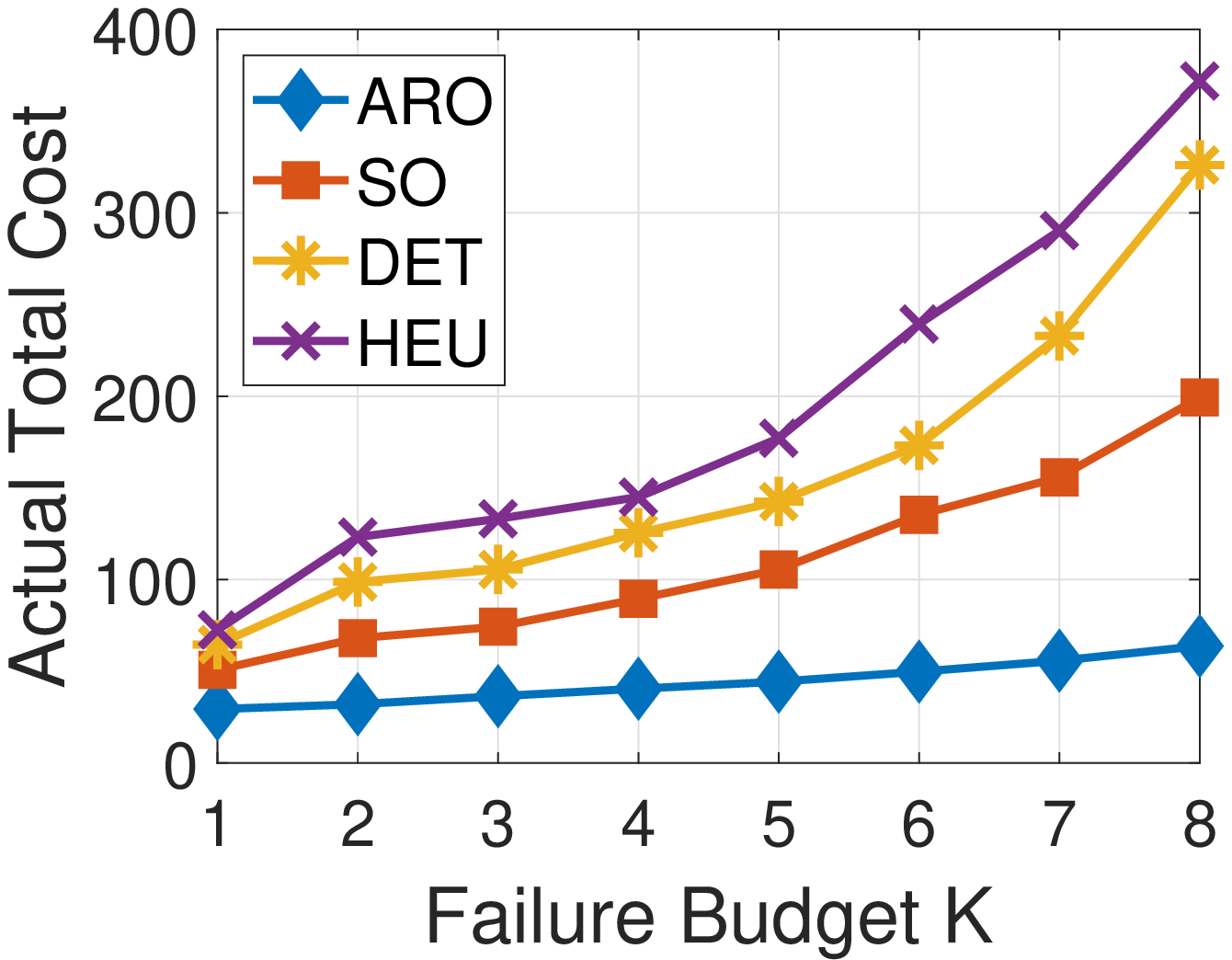}
	     \label{fig:worst_table}
	}  \vspace{-0.2cm}
	\subfigure[Average cost]{
	     \includegraphics[width=0.245\textwidth,height=0.10\textheight]{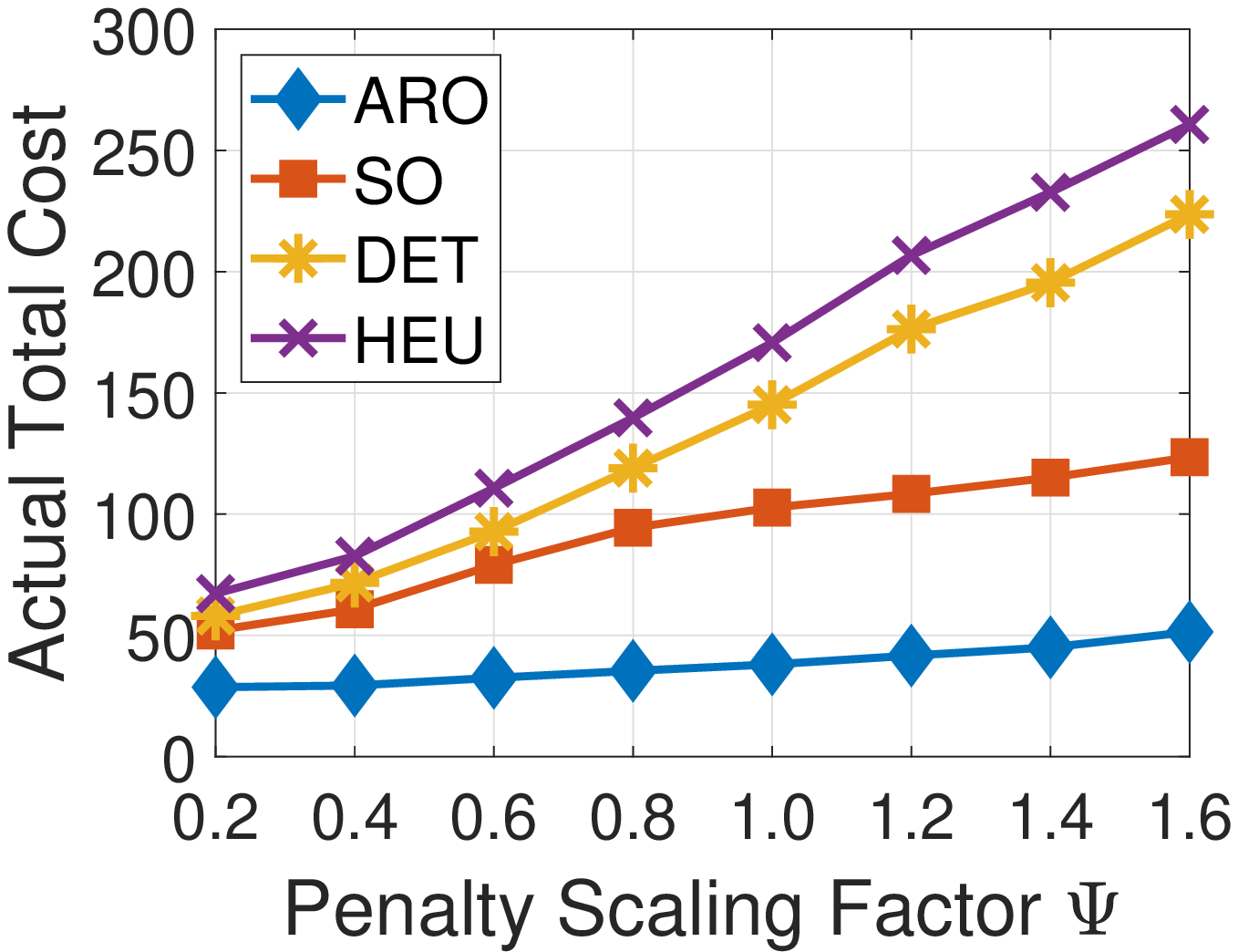}
	     \label{fig: avg_Penalty}
	}  \hspace*{-2.1em}
	\subfigure[Worst-case cost]{
	     \includegraphics[width=0.245\textwidth,height=0.10\textheight]{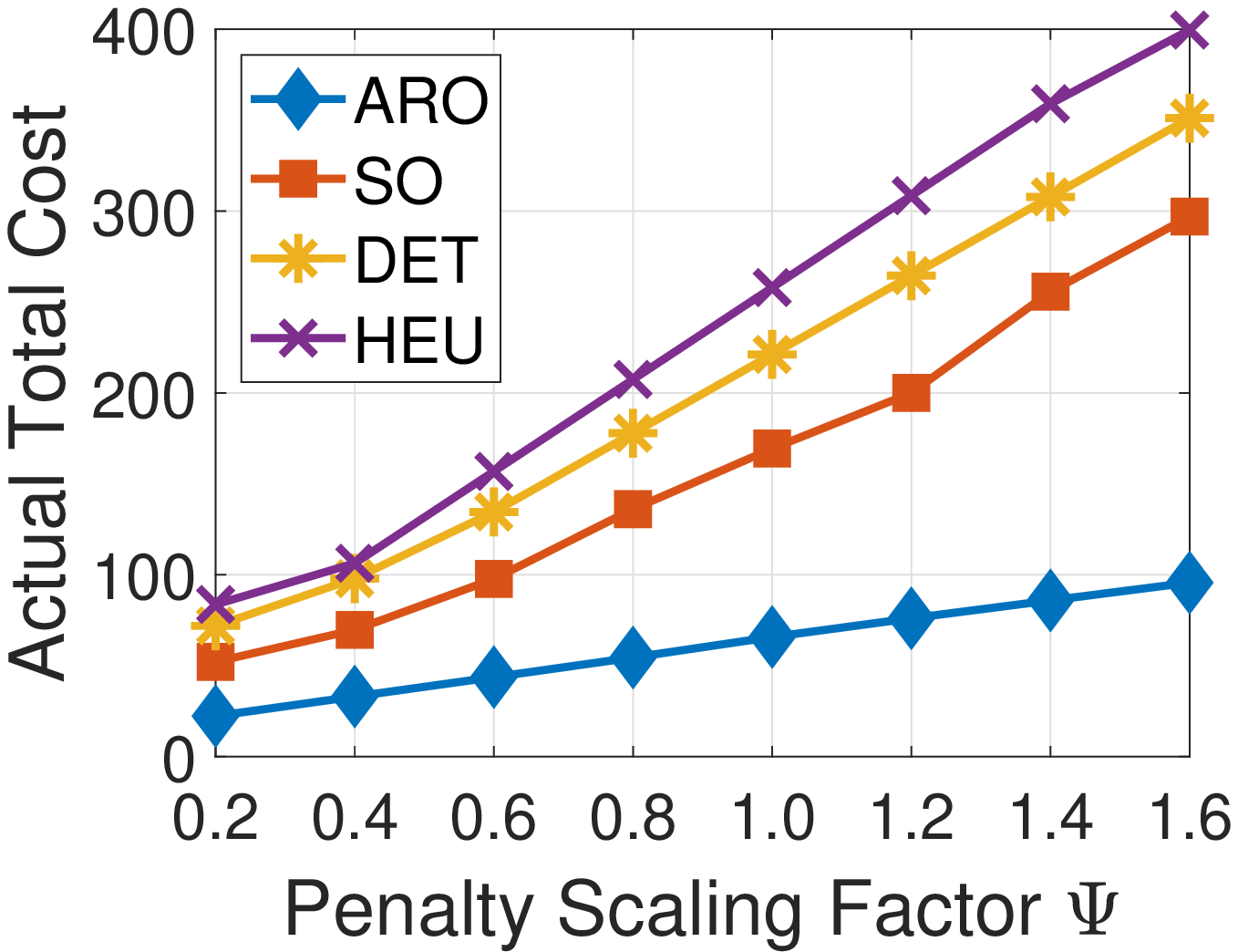}
	     \label{fig: Worst_Penalty}
	}  \vspace{-0.2cm}
	\caption{Performance comparison}
\end{figure} 

To illustrate the benefit of integrating failure uncertainty in the decision-making model of the SP, we will compare the performance of the proposed ARO model \textit{with} and \textit{without} failure consideration. For the case without failure consideration, the SP can simply set $K = 0$. For ARO with failure consideration, we set $K$ to be 2 (i.e., the first-stage decisions are robust against up to any two simultaneous failed ENs). Our results show that the provisioning costs with and without considering failures are almost the same. Both schemes almost exhaust the budget. On the other hand, Figs.~\ref{figure: avg_cost}-\ref{figure: worst_cost} show that ARO with failure consideration significantly reduces the SP's cost during node failure events. Although we do not present the result here, the benefit of considering failures  (i.e., the gap between two curves in each figure) increases drastically when the unmet demand penalty $P$ increases. Hence, failure consideration is important for SPs who need to maintain high service quality (i.e., less unmet demand). 

For ARO with $K = 2$, we can see that the cost increases slowly when the number of \textit{actual} failures is small, which implies that a minimal preparation 
can make the system resilient to unexpected failures. When more than five ENs (i.e., more than 25\% of the ENs) fail simultaneously, the cost increases faster because the probability that an EN selected for service placement in the first stage fails increases. To hedge against a large number of  failures, the SP  should set $K$ to be higher. However, it will increase the service provisioning cost. Hence, based on the desired level of robustness and resiliency, the SP needs to decide a proper value for $K$.
\begin{figure}[h!]
\label{fig:performance}
		\subfigure[Average cost]{
		  \includegraphics[width=0.245\textwidth,height=0.10\textheight]{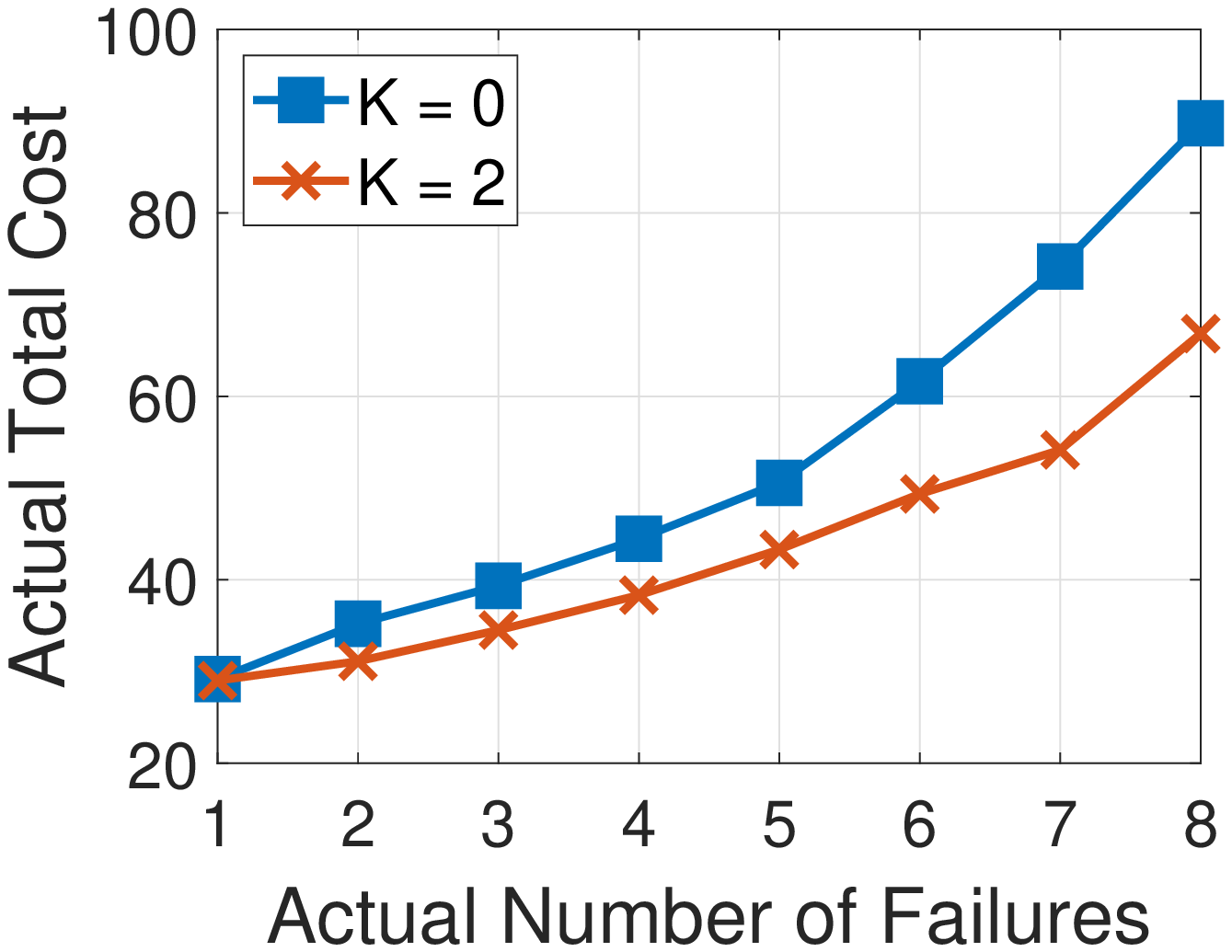}
	    \label{figure: avg_cost}
	}   \hspace*{-2.1em} 
		 \subfigure[Worst-case cost]{
	     \includegraphics[width=0.245\textwidth,height=0.10\textheight]{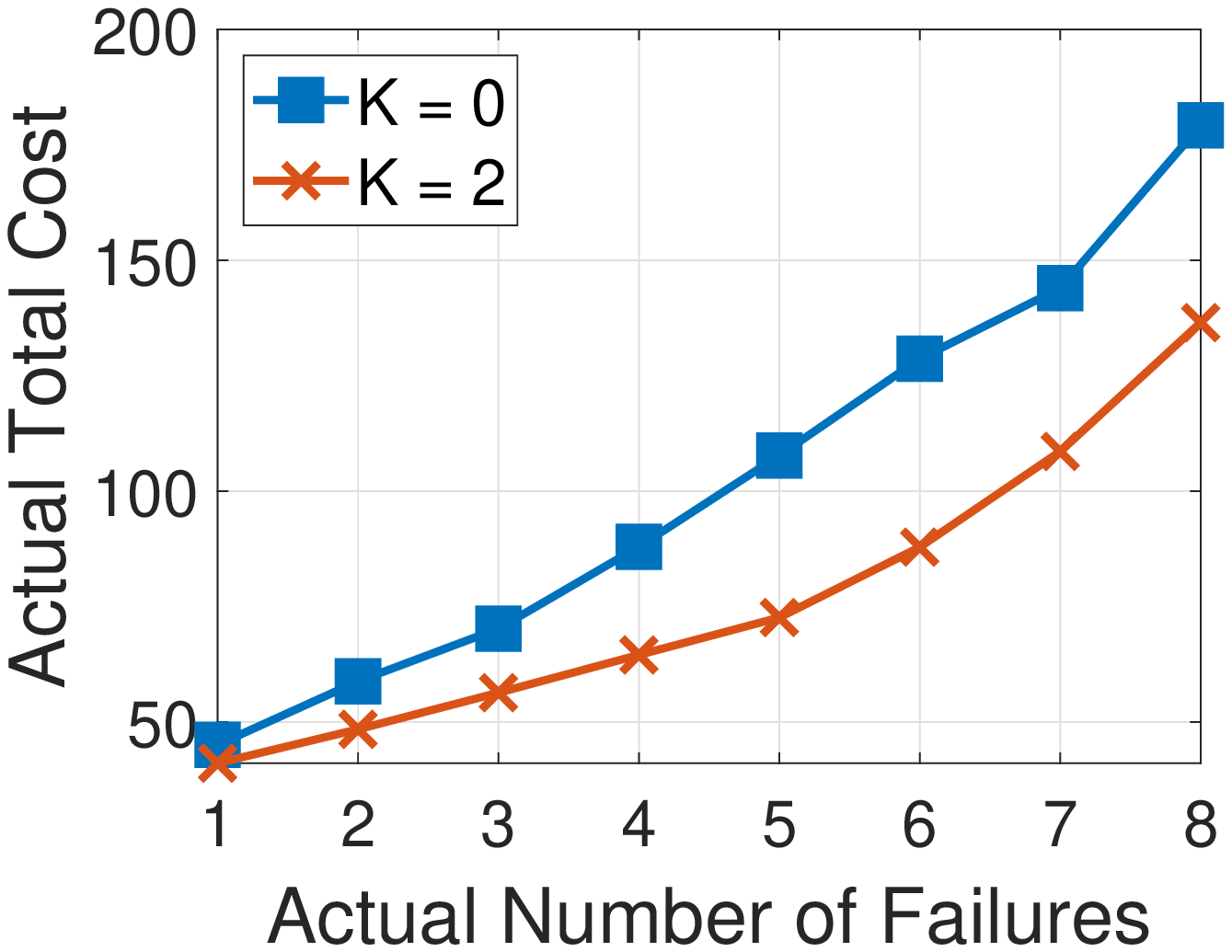}
	     \label{figure: worst_cost}
	}  \vspace{-0.2cm}
	\caption{The advantages of considering failure uncertainty}
\end{figure}

\subsection{Comparison Between \textbf{Algorithm \ref{CCGalg}} and ADR}

In the following, we analyze the performance of the CCG-based iterative algorithm (i.e., \textbf{Algorithm \ref{CCGalg}})  and the ADR approach. Fig.~\ref{fig:LDR_CCG} depicts the optimal costs produced by \textbf{Algorithm \ref{CCGalg}} and ADR when we vary the demand uncertainty budget $\Gamma$ and the failure budget $K$. Recall that \textbf{Algorithm \ref{CCGalg}} outputs an exact optimal solution for  $(\mathcal{P}_1$) while ADR gives an approximation solution. We can see that the total costs produced by the ADR policy are quite close to the exact optimal values obtained from \textbf{Algorithm \ref{CCGalg}}. The optimality gap is small when $K$ is small. In practice, the number of simultaneous failures is typically small, except during extreme events like natural disasters. Note that in our setting,  two failures already mean 10\% of nodes fail. Hence, ADR can be an alternative tool for the SP to solve most practical cases.  

\begin{figure}[h!]
\centering
	\includegraphics[width=0.30\textwidth,height=0.12\textheight]{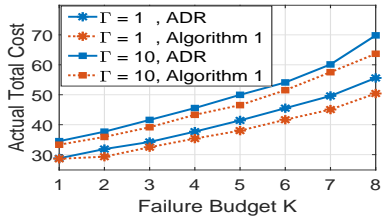}
	\caption{Comparison between \textbf{Algorithm 1} and ADR} 
	\label{fig:LDR_CCG}
\end{figure} 

\begin{table}[h]
\centering
\begin{tabular}{ |c||c|c|c|}
 \hline
 \multicolumn{4}{|c|}{Running time comparison ($\Gamma = 10$)} \\
 \hline
 $K$ & Duality-based CCG (s) & KKT-based CCG (s) & ADR (s) \\
 \hline
 5 & 53.60  & 3387s & 131.87\\
 \hline
 7 & 179.79 & 11876s & 135.84\\
 \hline
 9 & 629.24 & NA & 137.47\\
 \hline
11 & 576.89 & NA & 136.39\\
 \hline
\end{tabular}
\caption{Computational time comparison}
\label{tab: time_table}
\end{table}

\subsection{Time complexity analysis}
Table \ref{tab: time_table} reports the running time of the KKT-based CCG (see \textit{Appendix} \ref{kkta}  \cite{arxiv22}), the proposed duality-based CCG, and ADR algorithms for different values of $K$. We can see that the computational time of the CCG-based algorithms is sensitive to $K$; thus, they are sensitive to the uncertainty set. On the other hand, ADR is not very sensitive to the uncertainty set because ADR solves a MILP that does not depend on the number of extreme points of $\Xi$. Also, the KKT-based CCG algorithm is generally slower than the duality-based CCG algorithm since the KKT-based subproblem reformulation has a significantly larger size (in terms of the number of constraints and integer variables) compared to the duality-based subproblem reformulation. For example, when $K = 9$, the KKT-based CCG algorithm does not converge after 10 hours. 

We have also conducted experiments by varying $K$ and $\Gamma$ and found that \textbf{Algorithm \ref{CCGalg}} normally converges quickly for small values of $K$ and $\Gamma$. The duality-based CCG algorithm typically runs faster than the KKT-based CCG algorithm. For large uncertainty sets, instead of CCG, the SP may use ADR to obtain a good approximation solution in a reasonable time. 
In practice, the SP can run both the duality-based CCG and ADR algorithms in parallel and only implement the solution obtained by the ADR algorithm if the duality-based CCG does not converge after a certain amount of time. 

\begin{table}[h!]
\centering
\begin{tabular}{|c|l|clll|cll|}
\hline
Problem size &
  \multicolumn{1}{c|}{K} &
  \multicolumn{4}{c|}{Duality-based CCG (s)} &
  \multicolumn{3}{c|}{ADR (s)} \\ \hline
\multirow{4}{*}{\begin{tabular}[c]{@{}c@{}}I = J = 20\\ ($\Gamma =10$)\end{tabular}} &
  \multicolumn{1}{c|}{5} &
  \multicolumn{4}{c|}{66.15} &
  \multicolumn{3}{c|}{131.87} \\ \cline{2-9} 
 &
  \multicolumn{1}{c|}{7} &
  \multicolumn{4}{c|}{179.79} &
  \multicolumn{3}{c|}{135.84} \\ \cline{2-9} 
 &
  \multicolumn{1}{c|}{9} &
  \multicolumn{4}{c|}{529.24} &
  \multicolumn{3}{c|}{137.47} \\ \cline{2-9} 
 &
  11 &
  \multicolumn{4}{c|}{576.89} &
  \multicolumn{3}{c|}{138.39} \\ \hline
\multirow{4}{*}{\begin{tabular}[c]{@{}c@{}}I = 100 \\ J = 20\\ ($\Gamma = 10$)\end{tabular}} &
  \multicolumn{1}{c|}{5} &
  \multicolumn{4}{c|}{456.51} &
  \multicolumn{3}{c|}{3873.12} \\ \cline{2-9} 
 &
  \multicolumn{1}{c|}{7} &
  \multicolumn{4}{c|}{465.98} &
  \multicolumn{3}{c|}{3974.91} \\ \cline{2-9} 
 &
  \multicolumn{1}{c|}{9} &
  \multicolumn{4}{c|}{470.05} &
  \multicolumn{3}{c|}{3915.13} \\ \cline{2-9} 
 &
  \multicolumn{1}{c|}{11} &
  \multicolumn{4}{c|}{\begin{tabular}[c]{@{}c@{}} 465.21 (0.1\% gap) \\ 418.61 (0.5\% gap) \end{tabular}} &
  \multicolumn{3}{c|}{4015.87} \\ \hline
\multicolumn{1}{|l|}{\multirow{5}{*}{\begin{tabular}[c]{@{}l@{}}I = J = 25\\ ($\Gamma = 10$)\end{tabular}}} &
  5 &
  \multicolumn{4}{c|}{\begin{tabular}[c]{@{}c@{}}575.38 (0.5\% gap)\\ 537.19 (1\% gap)\end{tabular}} &
  \multicolumn{3}{c|}{1119.45} \\ \cline{2-9} 
\multicolumn{1}{|l|}{} &
  7 &
  \multicolumn{4}{c|}{\begin{tabular}[c]{@{}c@{}}1174.32 (0.5\% gap)\\ 960.52 (1\% gap)\end{tabular}} &
  \multicolumn{3}{c|}{1198.21} \\ \cline{2-9} 
\multicolumn{1}{|l|}{} &
  9 &
  \multicolumn{4}{c|}{\begin{tabular}[c]{@{}c@{}} 
  2937.03 (0.5\% gap)\\ 2617.98 (1\% gap)\end{tabular}} &
  \multicolumn{3}{c|}{1077.12} \\ \cline{2-9} 
\multicolumn{1}{|l|}{} &
  11 &
  \multicolumn{4}{c|}{\begin{tabular}[c]{@{}c@{}} 
  3125.17 (0.5\% gap)\\ 3032.22 (1\% gap)\end{tabular}} &
  \multicolumn{3}{l|}{1014.83} \\ \hline
\multirow{4}{*}{\begin{tabular}[c]{@{}c@{}}I = J = 30\\ ($\Gamma = 10$)\end{tabular}} &
  5 &
  \multicolumn{4}{c|}{\begin{tabular}[c]{@{}c@{}}7326.25 (0.1\% gap)\\ 
  5173.84 (1\% gap)\end{tabular}} &
  \multicolumn{3}{c|}{1167.69} \\ \cline{2-9} 
 &
  7 &
  \multicolumn{4}{c|}{\begin{tabular}[c]{@{}c@{}}11246.71 (0.1\% gap)\\ 
  8011.77 (1\% gap) 
  \end{tabular}} &
  \multicolumn{3}{c|}{1250.87} \\ \cline{2-9} 
 &
  9 &
  \multicolumn{4}{c|}{\begin{tabular}[c]{@{}c@{}}14789.21 (0.1\% gap)\\ 
  9933.21 (1\% gap) 
  \end{tabular}} &
  \multicolumn{3}{c|}{1315.14} \\ \cline{2-9} 
 &
  11 &
  \multicolumn{4}{c|}{\begin{tabular}[c]{@{}c@{}}15981.04 (0.1\% gap)\\ 
  9279.31 (1\% gap) 
  \end{tabular}} &
  \multicolumn{3}{c|}{1285.72} \\ \hline
\multirow{4}{*}{\begin{tabular}[c]{@{}c@{}}I = J = 40\\ ($\Gamma = 10$)\end{tabular}} &
  5 &
  \multicolumn{4}{c|}{\begin{tabular}[c]{@{}c@{}} 11326.42 (0.1\% gap) \\ 6903.76 (0.5\% gap) \\  1953.24 (5\% gap)\end{tabular}} &
  \multicolumn{3}{c|}{3279.71} \\ \cline{2-9} 
 &
  7 &
  \multicolumn{4}{c|}{\begin{tabular}[c]{@{}c@{}} 15032.17 (0.1\% gap)\\ 13685.35 (0.5\% gap) \\ 4317.21 (5\% gap)  \end{tabular}} &
  \multicolumn{3}{c|}{3365.3} \\ \cline{2-9} 
 &
  9 &
  \multicolumn{4}{c|}{\begin{tabular}[c]{@{}c@{}}17893.42 (0.1\% gap)\\ 14593.28 (0.5\% gap)\\ 5927.43 (5\% gap)\end{tabular}} &
  \multicolumn{3}{c|}{3476.82} \\ \cline{2-9} 
 &
  11 &
  \multicolumn{4}{c|}{\begin{tabular}[c]{@{}c@{}}25379.01 (0.1\% gap)\\ 16197.37 (0.5\% gap)\\   6834.94 (5\% gap)\end{tabular}} &
  \multicolumn{3}{c|}{3575.02} \\ \hline
\end{tabular}
\caption{Run-time experiments}
\label{table:runtime2}
\end{table}

Table \ref{table:runtime2} illustrates  the running time for varying network sizes and uncertainty sets. Additional computational results with $\Gamma = 5$ are provided in \textit{Appendix} \ref{appendix:gamma5}  \cite{arxiv22}. 
For each setting, we take the average time over 10 different problem instances.  As expected, the computational time of the ADR reformulation is significantly influenced by the network size but remains relatively unaffected by the uncertainty set. For the duality-based CCG algorithm,  we set the (optimality) gap between the UB and LB in \textbf{Algorithm 1} to be 0.1\%, 0.5\%, 1\%, and 5\%. The algorithm can deliver a solution within a reasonable time even, for large-scale networks. The running time of the duality-based CCG algorithm is heavily dependent on the uncertainty set. Furthermore, it exhibits a greater sensitivity to the number of ENs ($J$) rather than the number of areas. It is because $J$ affects the number of integer variables in the MP. It is worth emphasizing that the underlying problem is a planning problem that 
does not require real-time optimization. 

\subsection{Convergence of \textbf{Algorithm \ref{CCGalg}}}
Figs.~\ref{fig:convergence1}-\ref{fig:convergence2} show the convergence of  \textbf{Algorithm \ref{CCGalg}} for different values of $K$ and $\Gamma$. 
The algorithm converges within a small number of iterations in both cases. However, the number of iterations increases as $K$ and $\Gamma$ increase. Note that  larger values of $K$ and $\Gamma$ imply the uncertainty set $\Xi$ has more extreme points. 
Thus, the result confirms that the convergence \textbf{Algorithm \ref{CCGalg}} is sensitive to the uncertainty set.

\begin{figure}[h!]
		\subfigure[K = 2, $\Gamma$ = 5]{
		  \includegraphics[width=0.245\textwidth,height=0.10\textheight]{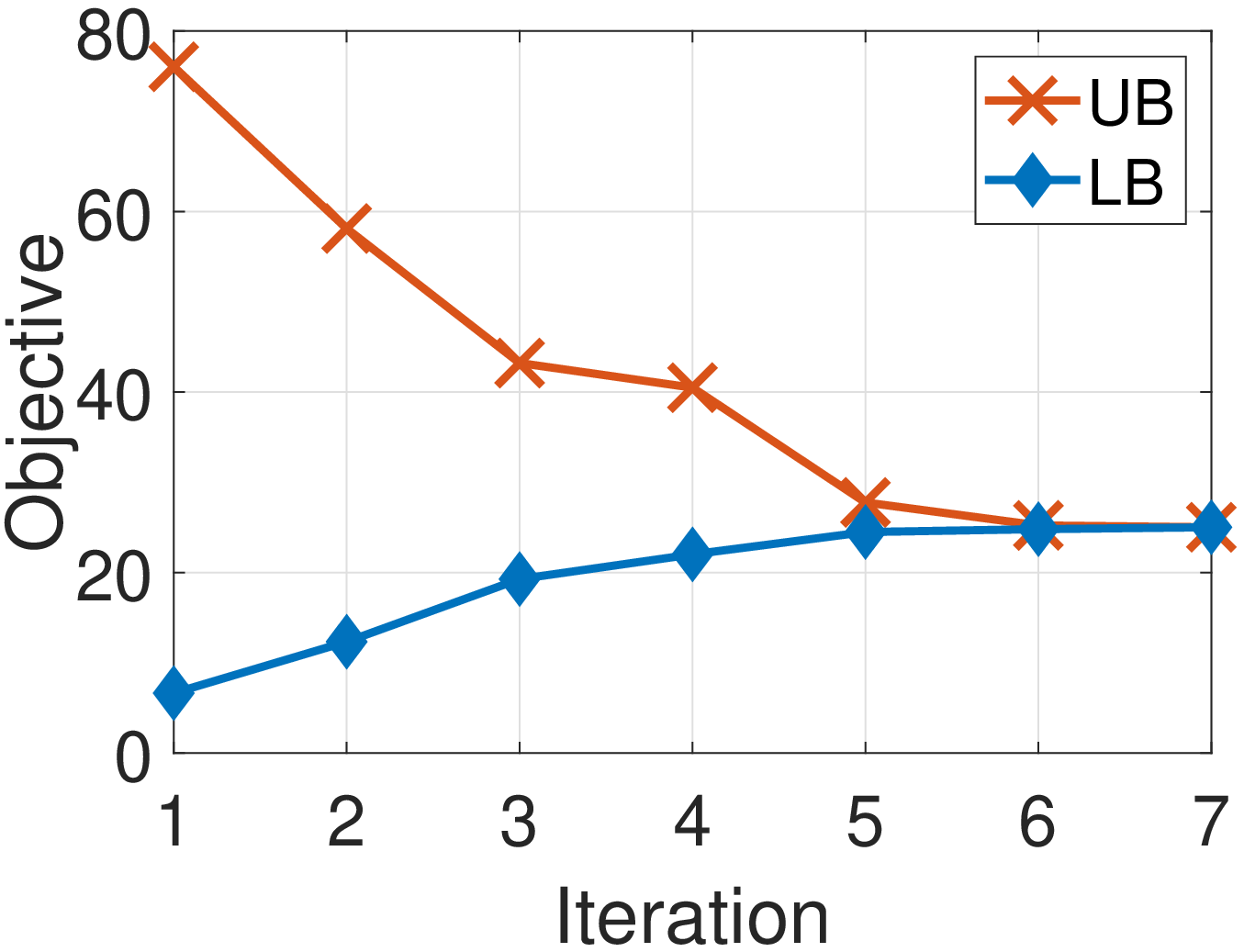}
	    \label{fig:convergence1}
	}   \hspace*{-2.1em} 
		 \subfigure[K = 5, $\Gamma$ = 10]{
	     \includegraphics[width=0.245\textwidth,height=0.10\textheight]{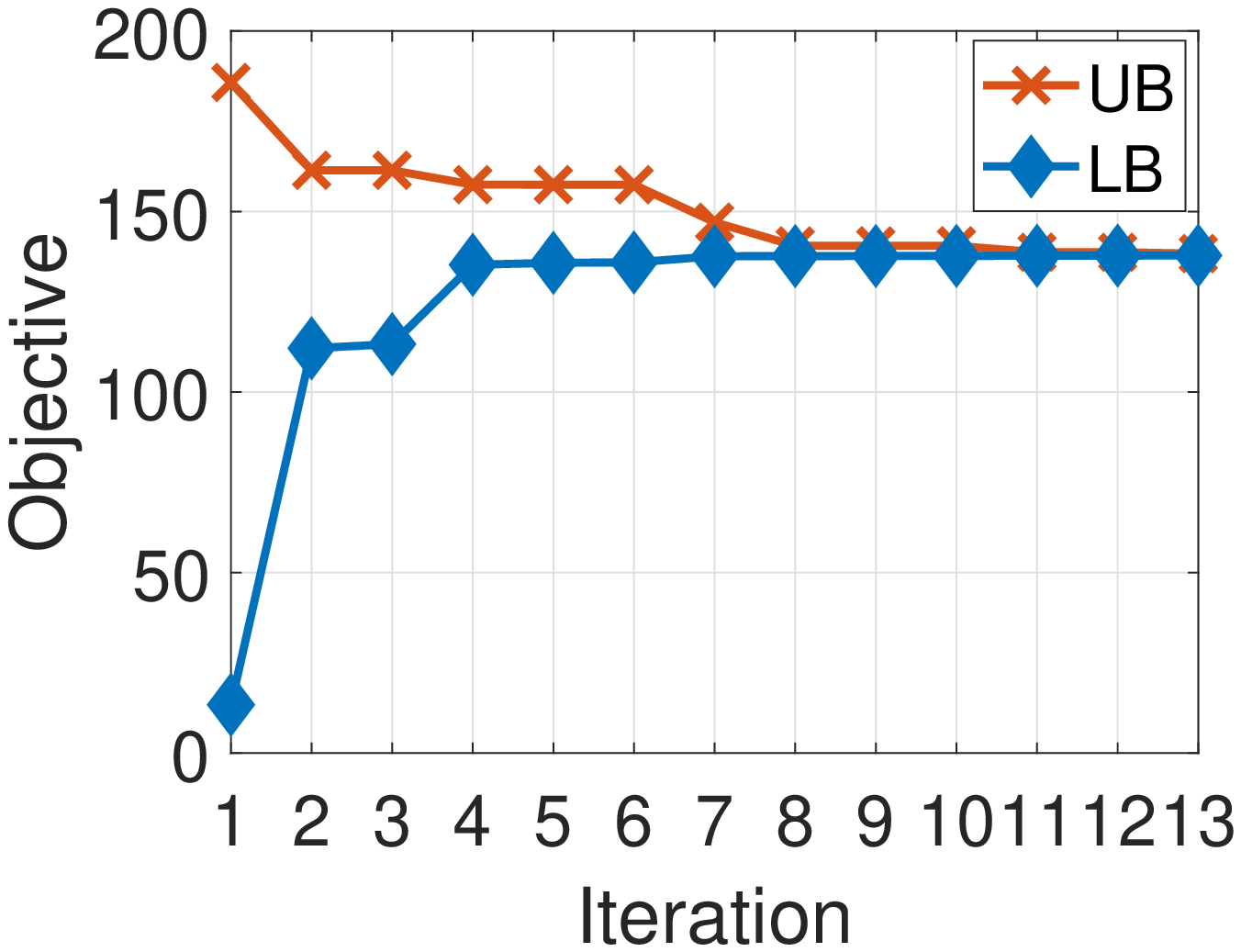}
	     \label{fig:convergence2}
	}  \vspace{-0.2cm}
	\caption{Convergence of \textbf{Algorithm \ref{CCGalg}}}
\end{figure}

\vspace{-0.3cm}

\subsection{Sensitivity Analysis}
\label{sensi}

We now evaluate the impacts of important system parameters on the optimal solution. Figs.~\ref{fig: size_of_K}-\ref{fig: size_of_Gamma} illustrate the impacts of the uncertainty set on the system performance. Note that \textit{``Objective''} is the optimal objective value of ($\mathcal{P}_1$).  As expected, the total cost increases as the uncertainty set enlarges (i.e., when $K$ and $\Gamma$ increase). Hence, there is an inherent trade-off between the system cost and the level of robustness. We can also see that the system cost increases when the delay penalty $\beta$ increases (i.e., the SP is more delay-sensitive). Additionally, Figs. \ref{fig: alpha_psi_cost}-\ref{fig: alpha_beta_cost} further show that the cost increases as $\alpha$ increases. 
This is because a higher value of $\alpha$ implies a larger uncertainty set, which consequently leads to a higher cost for a more robust and conservative solution. Fig.~\ref{fig: alpha_psi_cost} illustrates that the total cost increases as the unmet demand penalty increases.

\begin{figure}[h!]
	\subfigure[Varying $\beta$ and $K$]{
		  \includegraphics[width=0.24\textwidth,height=0.10\textheight]{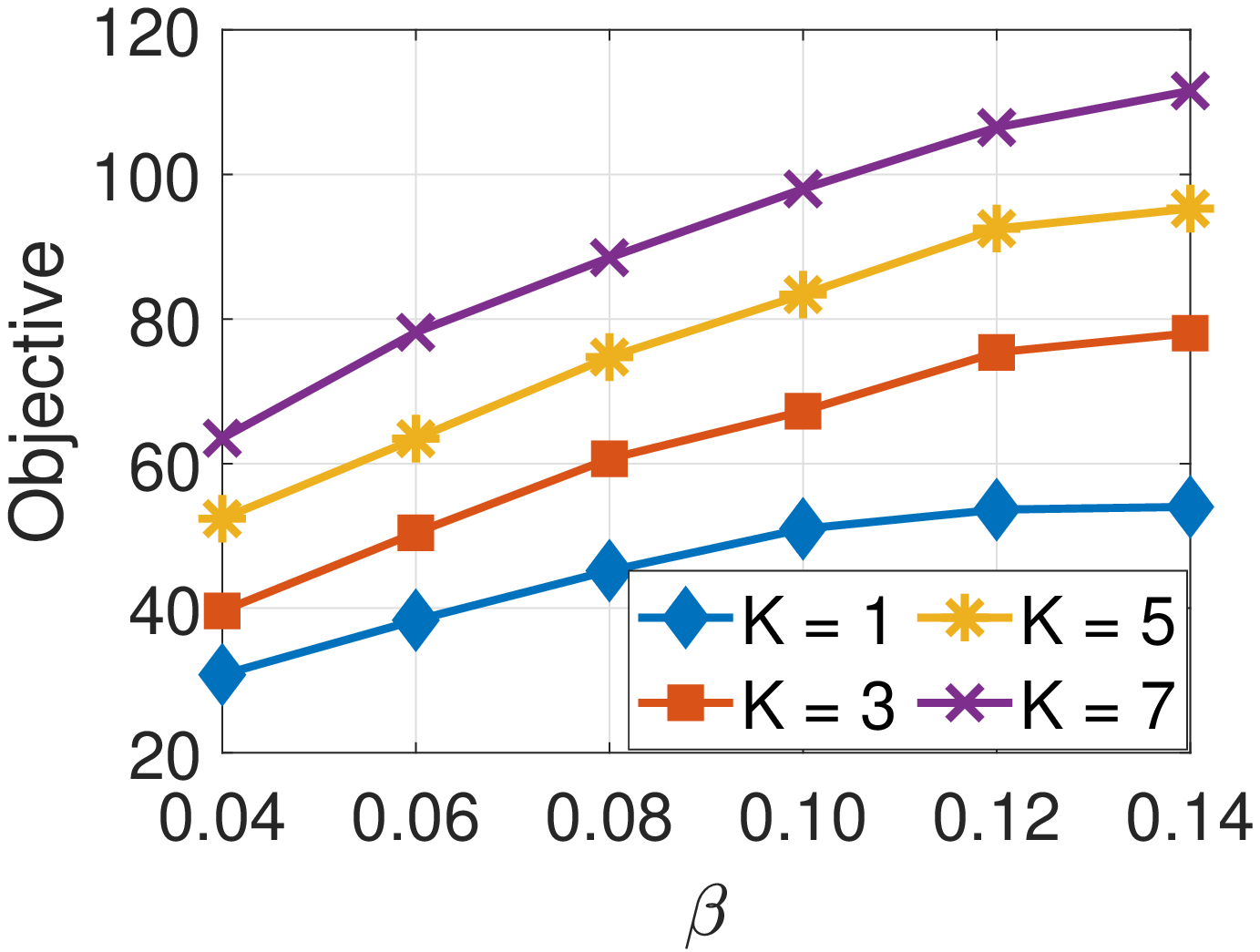}
	    \label{fig: size_of_K}
	}   \hspace*{-2.1em} 
		 \subfigure[Varying $\beta$ and $\Gamma$]{
	     \includegraphics[width=0.24\textwidth,height=0.10\textheight]{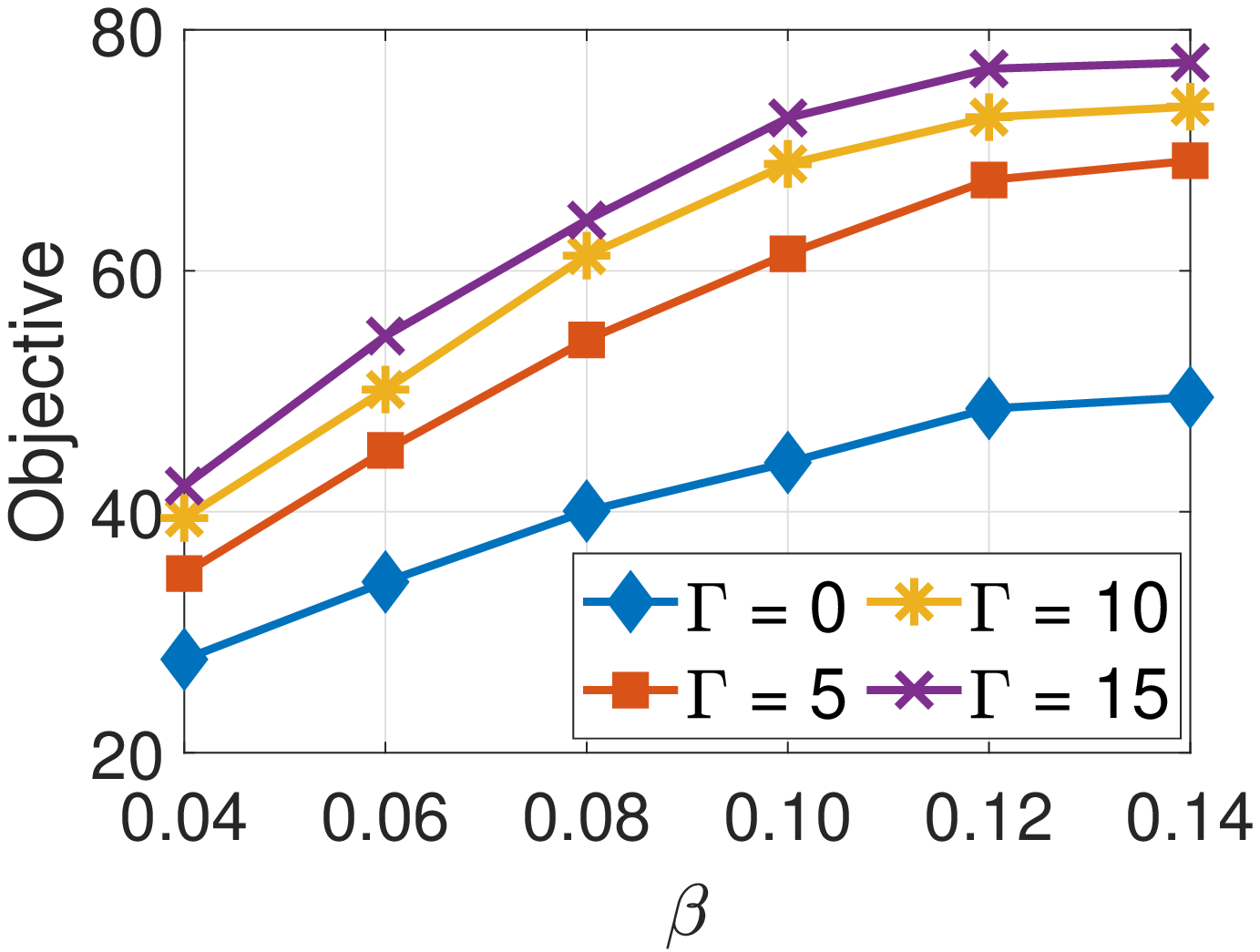}
	     \label{fig: size_of_Gamma}
	}  \vspace{-0.2cm}
	\subfigure[Varying $\alpha$ and $\Psi$]{
 	    \includegraphics[width=0.24\textwidth,height=0.10\textheight]{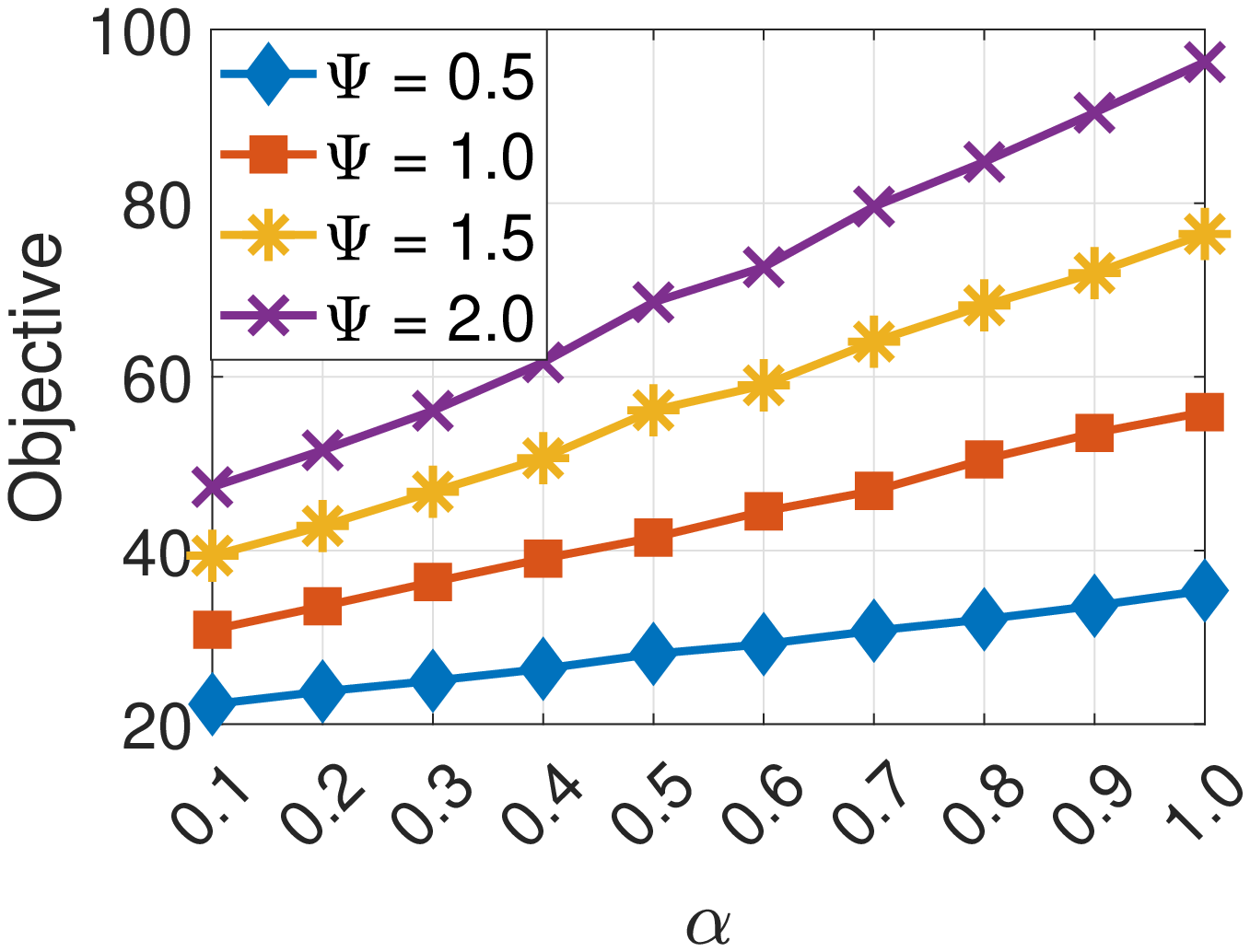} 
 	    \label{fig: alpha_psi_cost}} 
 	\hspace*{-2.1em} 
	\subfigure[Varying $\alpha$ and $\beta$]{
	     \includegraphics[width=0.24\textwidth,height=0.10\textheight]{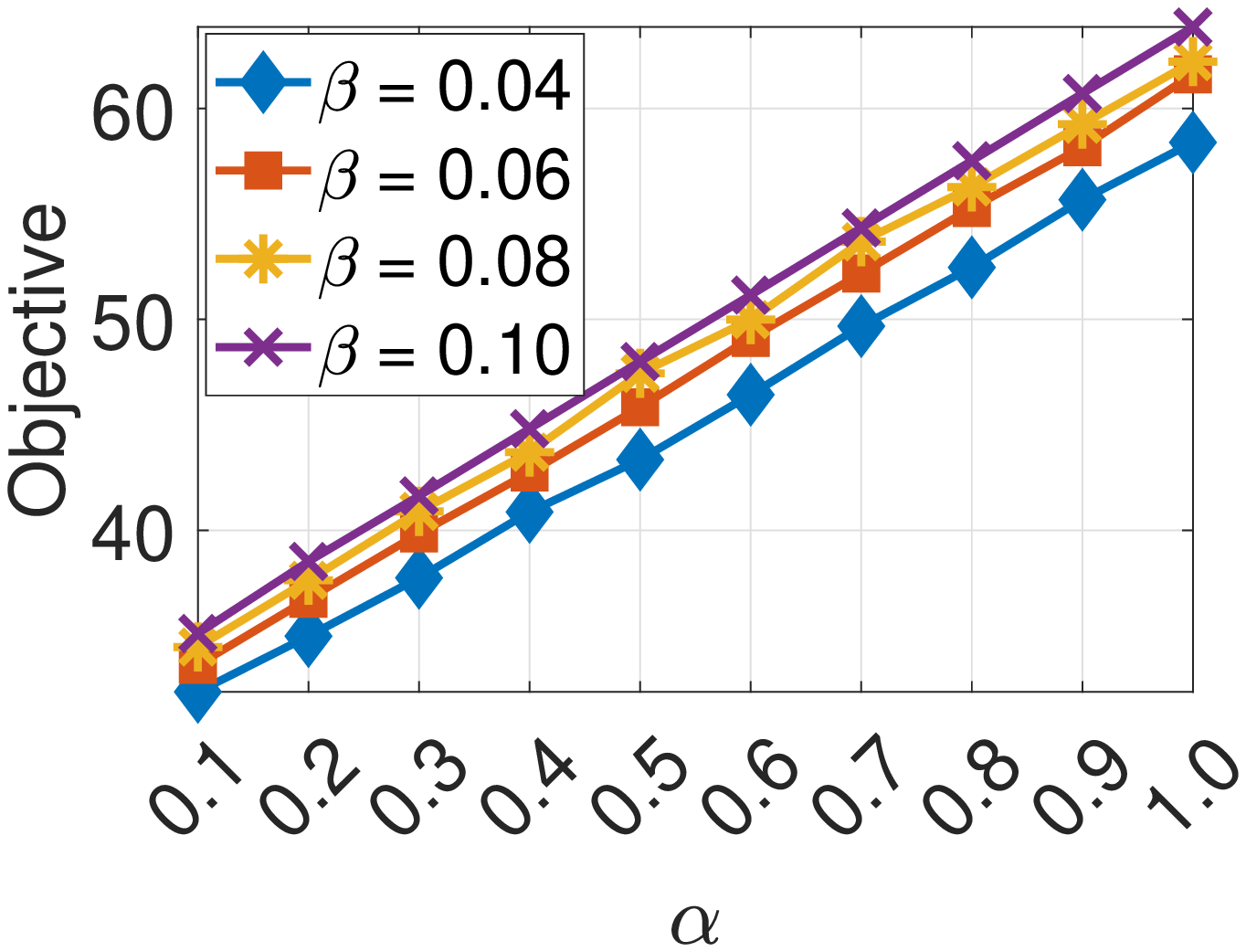}
	     \label{fig: alpha_beta_cost}
	} 
	\caption{The impacts of the uncertainty set on the performance}
\end{figure}

The impacts of the delay penalty parameter $\beta$ and the unmet demand penalty parameter $\Psi$ on the optimal solution are further demonstrated in Figs. \ref{fig: psi_beta_provisioning}-\ref{fig: psi_beta_cost}. We can observe that the objective value and provisioning cost increase as $\Psi$ and  $\beta$ increase. In addition, a higher value of $\beta$ indicates the SP is more delay sensitive and results in a higher cost. The SP may try to allocate the demand in every area to its closest ENs to reduce the delay and avoid the unmet demand penalty. Thus, cheaper ENs may not be chosen. According to Fig. \ref{fig: psi_beta_provisioning}, the SP is more willing to procure more resources in the first stage to reduce the delay and unmet demand penalties during the operation stage. The provisioning cost becomes saturated as the unmet demand penalty exceeds a certain threshold for each value of $\beta$. This implies that the SP fully spends its budget for service placement and resource procurement to enhance the service quality if $\beta$ and $\Psi$ (i.e., $P$) are sufficiently large. 

\begin{figure}[h!]
	\subfigure[Varying  $\beta$ and $\Psi$]{
 	    \includegraphics[width=0.242\textwidth,height=0.11\textheight]{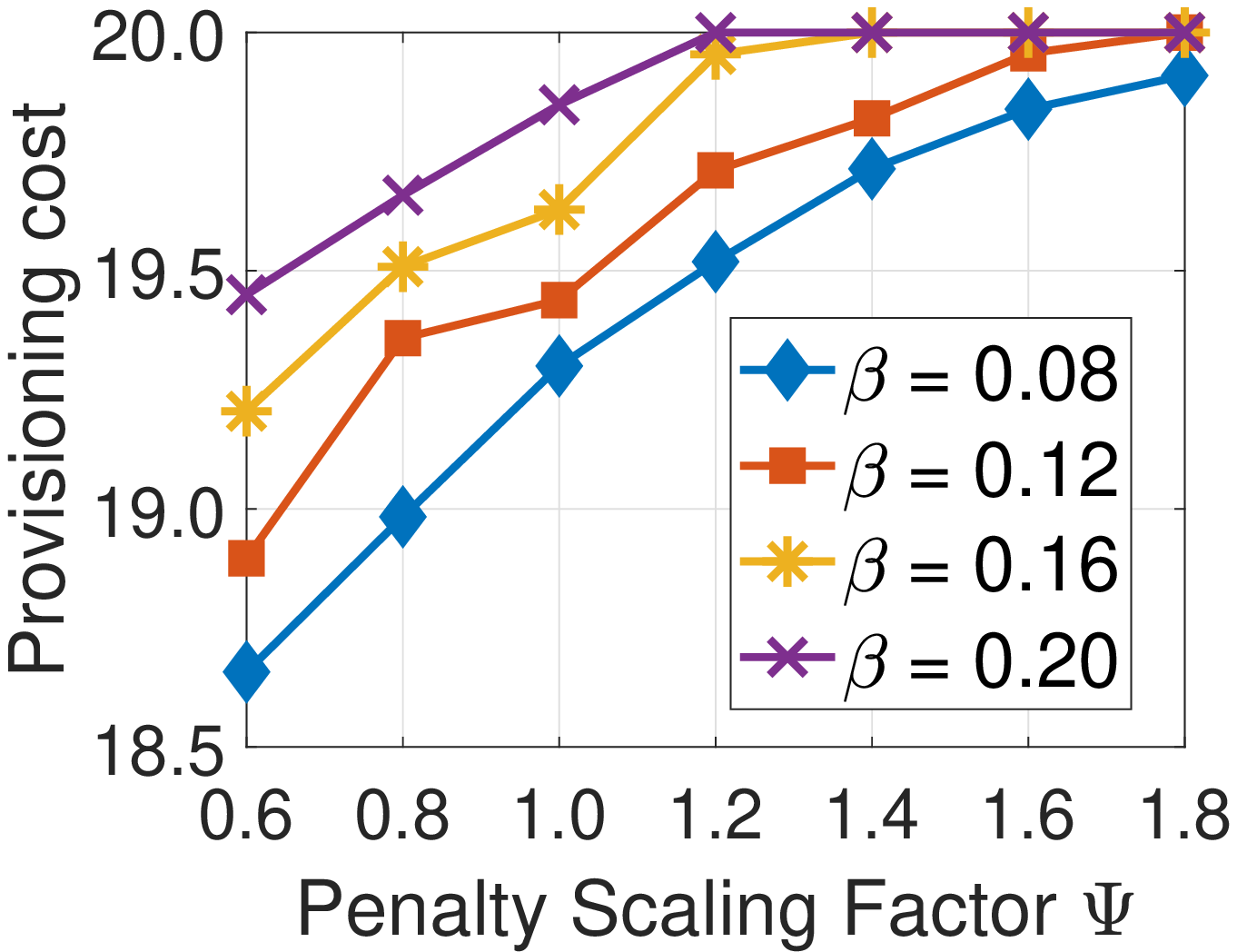} 
 	    \label{fig: psi_beta_provisioning}} 
 	\hspace*{-1.9em}
    \subfigure[Varying $\beta$ and $\Psi$]{
	     \includegraphics[width=0.242\textwidth,height=0.11\textheight]{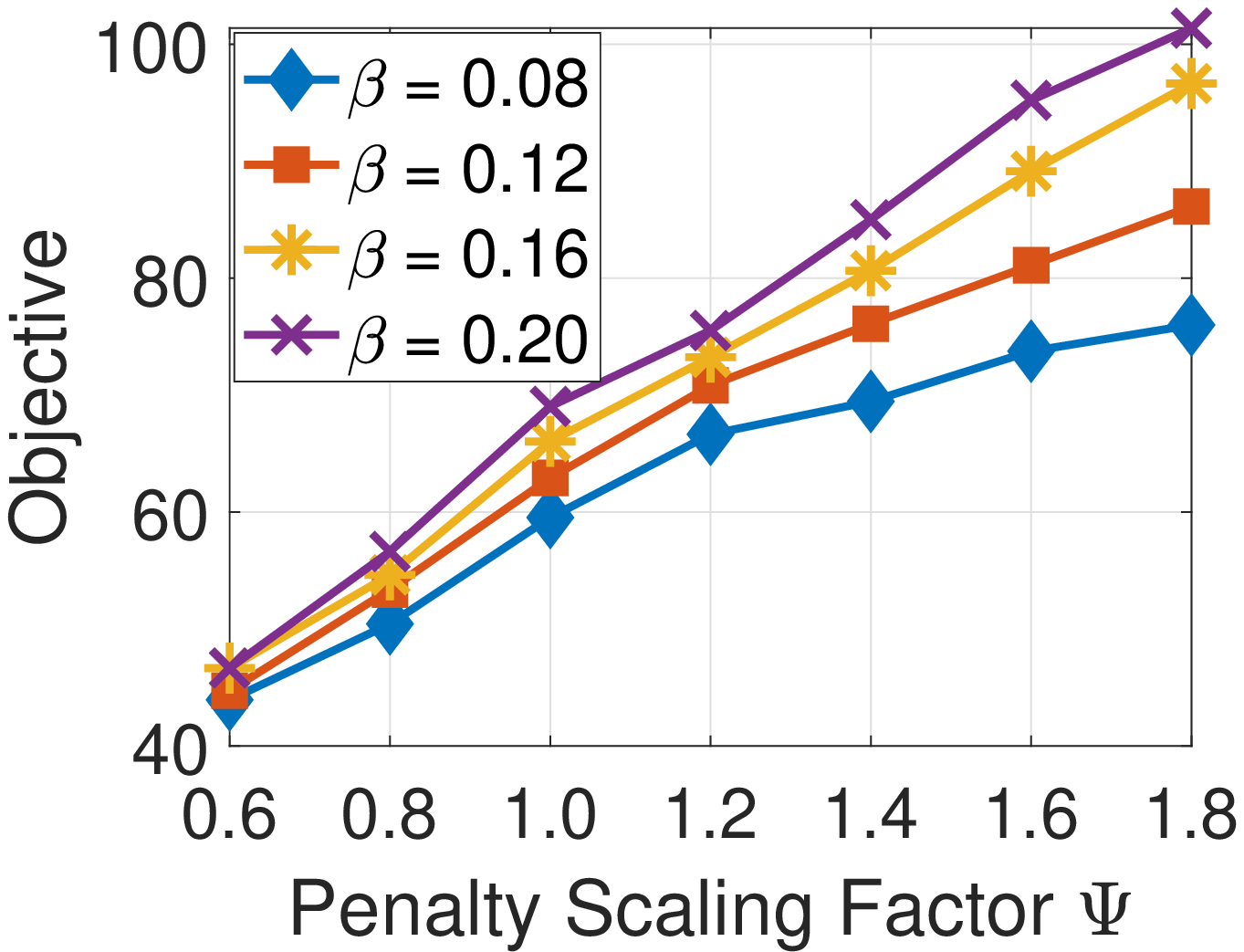}
	     \label{fig: psi_beta_cost}
	}   
	\caption{The impacts of $\beta$ and $\Psi$ on the system performance}
\end{figure}

Next, we examine the impacts of the system size on the optimal solution. Intuitively, when the number of ENs decreases, the total edge resource capacity is reduced. When the number of areas increases, the total demand in the system increases. Fig.~\ref{figure: Number_of_ENs} shows that the total cost decreases as the number of ENs increases because the SP has more 
flexibility for purchasing resources when more ENs are available. Also, as can be seen in Fig.~\ref{figure: Number_of_Area}, the total cost increases as the number of areas increases. This result is intuitive since the SP needs to procure more resources to serve higher demand.

\begin{figure}[h]
\label{fig:size_of_system}
		\subfigure[Varying number of ENs]{
		  \includegraphics[width=0.245\textwidth,height=0.10\textheight]{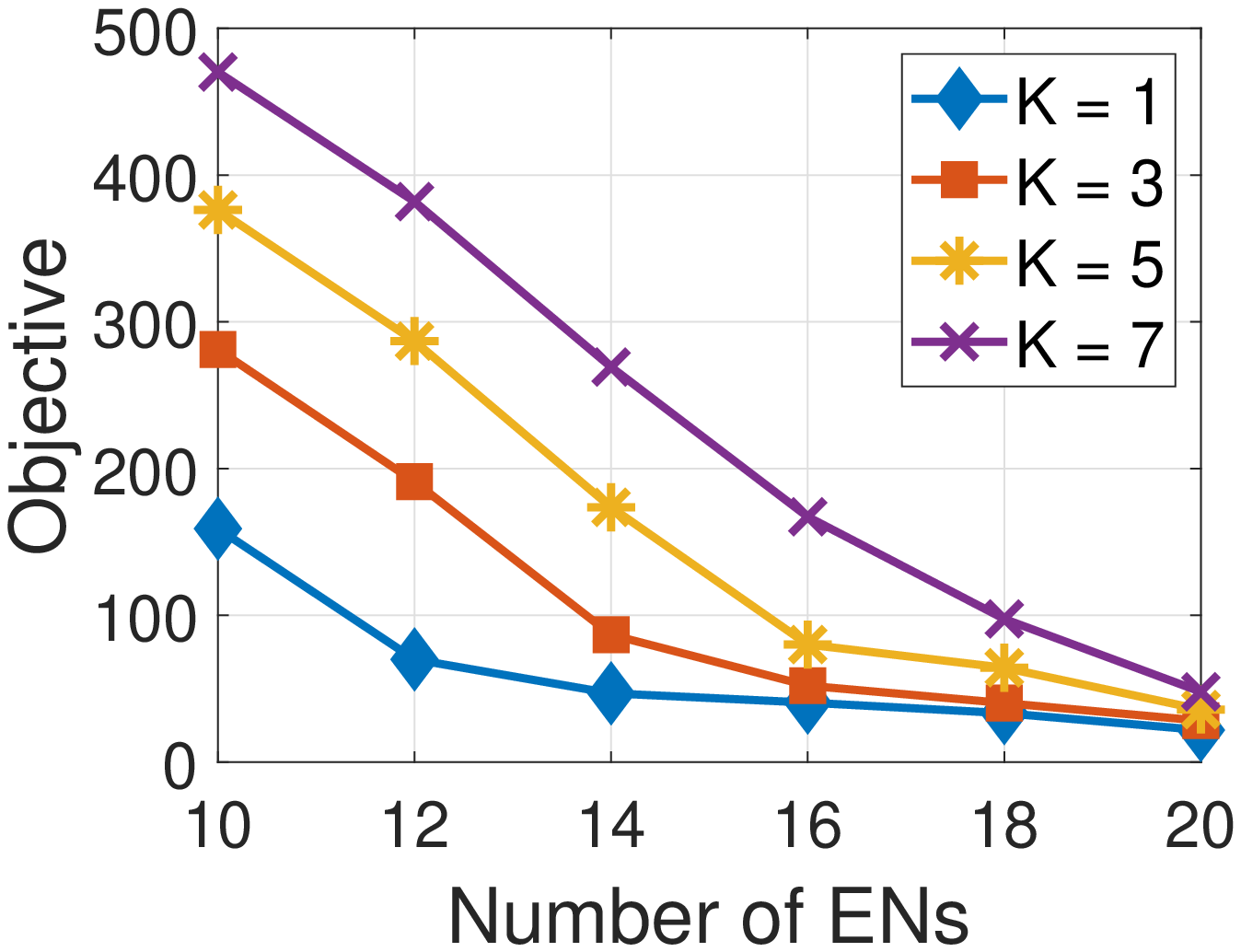}
	    \label{figure: Number_of_ENs}
	}   \hspace*{-2.05em} 
		 \subfigure[Varying number of areas ]{
	     \includegraphics[width=0.245\textwidth,height=0.10\textheight]{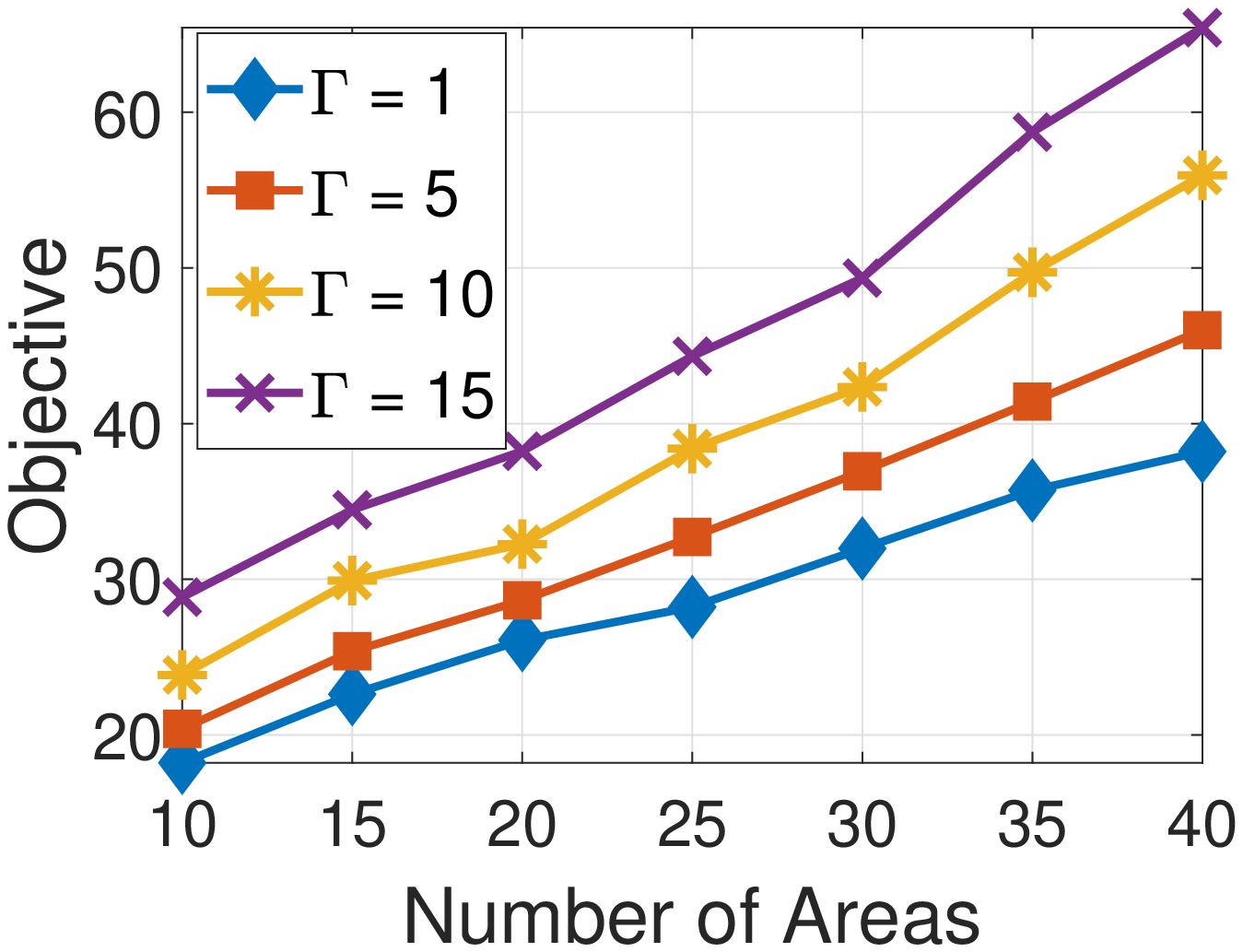}
	     \label{figure: Number_of_Area}
	}  \vspace{-0.2cm}
	\caption{The impacts of the system size on the performance}
\end{figure}

Fig. \ref{fig:delay_threshold} illustrates the impact of the maximum delay threshold (i.e., $D^{\sf max}$) and delay penalty $\beta$ on the system performance. As described in (\ref{avaliblity}), an EN $j$ can only serve user requests from area $i$ (i.e., $a_{i,j}= 1$) only if the delay between them is within the threshold $D^{\sf max}$. Hence, a decrease in $D^{\sf max}$, which indicates a more stringent delay requirement, leads to an increase in the number of $a_{i,j}$ values that become zero, indicating a reduction in the number of eligible ENs to serve user demand.  Consequently, the total cost of the system increases as $D^{\sf max}$ decreases and decreases as it increases. 

\begin{figure}[h!]
\centering
	\includegraphics[width=0.28\textwidth,height=0.11\textheight]{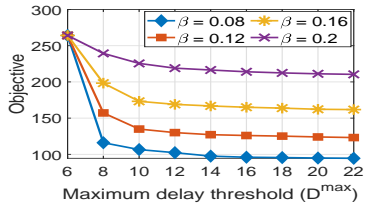}
	\caption{Impact of the maximum delay threshold $D^{\sf max}$} 
	\label{fig:delay_threshold}
\end{figure}

\section{Related Work}
\label{related_work}

There exists a substantial body of literature on edge resource allocation and service placement.
The problem of joint allocation of communication and computational resources for task offloading in  EC is  popular among the wireless community \cite{MEC17}. In \cite{duongtcc,duong1}, authors propose a market equilibrium approach for fair and efficient edge resource allocation. 
Farhadi \textit{et al.} \cite{vfar21} propose a two time-scale optimization model for the joint service placement and request scheduling problem under multi-dimensional resource and budget constraints. In \cite{spas19}, the authors present an optimal service placement solution that maximizes the total user utility. An IoT application provisioning problem is formulated in \cite{ryu19} to jointly optimize application placement and data routing, considering both bandwidth and delay requirements. However, these existing works primarily focus on deterministic models where all system parameters are assumed to be precisely known. A privacy-preserving decentrailized mobile edge content caching and sharing problem was formulated formulated in \cite{crowdcache} through non-cooperative games while considering the high mobility of MEDs.

A growing literature has studied edge resource management under uncertainty \cite{duong3,touy19,hbad20,jli21,Li21,duongiot,bandit_Netsoft}. Reference \cite{duong3} presents a primal-dual online LP method for matching multiple ENs and multiple services arriving in an online fashion, while \cite{bandit_Netsoft} presents a bandit-based online posted pricing scheme for allocating heterogeneous edge resources. In \cite{touy19}, the authors formulate the dynamic service placement problem as a contextual multi-armed bandit problem and propose a Thompson-sampling algorithm to properly allocate edge resources. Badri \textit{et al.} \cite{hbad20} cast an energy-aware multi-service placement problem under demand uncertainty as  
a stochastic problem, which aims to maximize the service quality under the energy budget constraint. Li \textit{et al.} \cite{jli21} employ an approximation algorithm to place the service and allocate resources considering demand uncertainty. Distributionally robust optimization is utilized in \cite{Li21} to tackle the edge service provisioning problem, with the goal of maximizing the expected revenue under the worst-case distribution of the uncertain demand. In \cite{duongiot}, 
an iterative algorithm is proposed to address the joint edge service placement and sizing problem under demand uncertainty. This line of research mainly  optimizes the system performance under demand uncertainty, without considering network failures. 

Extensive research has been conducted on the topic of reliable and availability-aware Network Function Virtualization (NFV) placement. Cziva \textit{et. al} \cite{infocom18} formulate a novel VNF placement problem for allocating VNFs to distributed ENs to minimize the total delay from all users to the associated VNFs. The joint VNF placement and routing problem is studied in \cite{Racha18} to minimize the queuing delay. A Bender-decomposition-based algorithm is introduced in \cite{pzha18}  to minimize the expected operational cost for 
resilient NFV placement. Reference \cite{Tao20} presents a MILP formulation for  joint VNF placement and scheduling to minimize the total cost while satisfying various latency requirements. A joint VNF placement and resource utilization problem is formulated in \cite{infocom20}, aiming to minimize the overall resource consumption of servers and links. In \cite{Nattakorn22}, the authors focus on the joint VNF placement and scheduling problem for profit maximization of delay-sensitive services, considering stringent service deadlines. Most of the existing NVF placement models are either deterministic or probabilistic, and predominantly employ   heuristic solution approaches. 

Resiliency and reliability in EC have also attracted a lot of attention recently. In \cite{aara21}, the authors evaluate the failure resilience of a service deployed on the edge infrastructure by learning the Spatiotemporal dependencies between edge server failures and exploiting the topological information to incorporate link failures. Kherraf \textit{et al.} \cite{Kherraf19} formulate a MIP for optimizing workload assignment to different ENs, considering a probabilistic reliability model for the ENs. Chemodanov \textit{et al.} \cite{dche20} use chance-constrained programming to study the problem of reliable orchestration of latency-sensitive SFCs, considering probabilistic capacity constraints. 
In \cite{mjoh15}, a single-stage robust backup resource network model is introduced to protect the system from random link failures. 
In \cite{sbia21}, the authors cast the distributed service chain composition problem with resource failures as a non-cooperative game and employ a weighted potential game to find an optimal Nash equilibrium to reduce request latency and network congestion. 

In \cite{risk_ton20}, the authors present a novel approach to the task offloading problem in a multi-server multi-access edge computing (MEC) environment, considering the potential failure of MEC servers. Their approach combines prospect theory and tragedy of the commons and involves multiple users optimizing their offloading decisions among the MEC servers, where each server constitutes a common pool of resources.Qu \textit{et al.} \cite{Qu2021} utilize the monotone sub-modular property to develop approximation algorithms to maximize the expected utility from serving user requests at the edge in the presence of uncertain service failures. Most existing works rely on probabilistic models to capture failure uncertainties, which can be difficult to obtain in practice. Additionally, they often consider only failure uncertainty without considering other uncertain system parameters such as resource demand. Furthermore, existing work primarily focuses on single-stage robust optimization and overlooks proactive decisions. To the best of our knowledge, we are the first to jointly consider demand and node failure uncertainties in a unified two-stage robust optimization framework, which helps the SP make proactive decisions to mitigate the impact of the uncertainties. Our work is different from the existing literature from both the modeling aspects and solution approaches.

\section{Conclusion and Future Work}
\label{conclusion}

In this paper, we proposed a two-stage robust optimization model for resilience-aware service placement in EC,
considering the uncertainties of both resource demand and EN failures. The proposed model helps the SP not only improve the user experience but also reduce its cost. 
We first developed a CCG-based iterative algorithm to find an exact optimal solution to the underlying robust problem. Due to fewer integer variables and constraints in the reformulated problem, the developed duality-based reformulation can help significantly speed up the computational time compared to the KKT-based reformulation method employed in the traditional CCG approach. 
To tackle large-scale problems, 
we further introduced a novel ADR-based approximation method, which is insensitive to the size of the uncertainty set, for solving the  problem.
The conducted experiments demonstrated the advantages of the proposed two-stage adaptive robust model compared to the deterministic, stochastic, and heuristic methods. Furthermore, the results show the considerable advantages of proactively preparing for unpredictable future events at a small extra provisioning cost, which encourages the SP to invest in improving resiliency. 

There are several interesting directions for future work. First, we would like to integrate other types of failures, such as link failures and partial EN failures, into the proposed robust model. We are also interested in considering other uncertainty factors, such as edge resource prices and renewable energy generation. Additionally, we would like to extend the model to consider VNF and multiple resource types, as well as multi-period workload allocation.
Finally, another interesting direction is to study the robust edge network design problem from the perspective of an edge infrastructure provider.

\bibliographystyle{IEEEtran}


\appendix

\subsection{Convergence of Algorithm \ref{CCGalg}}
\label{proofcon}
Because $\Xi^{\sf e}$ has a finite number of $R$ elements, to prove the \textit{Proposition \ref{prop1}}, we only need to show that before convergence, the subproblem always outputs a distinct extreme point at every iteration. In other words, if any extreme point $(\lambda^*, z^*)$ is repeated at two iterations, then $LB = UB$, which means  convergence. We can show it as follows.

Assume  $(y^*, t^* ,\eta^* )$ is an optimal solution to the \textit{MP} and $(\lambda^*, z^*, x^*, q^*)$ is an optimal solution to the subproblem in iteration $r$. Also, the extreme point $(\lambda^*, z^*)$ has appeared in a previous iteration. From the definition of the UB in (\ref{UBu}), we have:
\begin{align}
\label{eq:UB}
 UB \leq   \sum_j p_{j} y_{j}^* + \sum_j h_j t_j^* + \sum_i P_i q_i^{*} + \beta \sum_{i,j} d_{i,j} x_{i,j}^{*} .
\end{align}
Since $(\lambda^*, z^*)$ appeared in a previous iteration, the cuts related to $(\lambda^*, z^*)$ to be added to the \textit{MP} at iteration $r+1$
has already been added previously to the \textit{MP}. Thus,
the \textit{MP} in iteration $r+1$ is identical to \textit{MP} in iteration $r$. Hence, $(y^*, t^* ,\eta^* )$ is also the optimal solution to the \textit{MP} in iteration $r+1$. Note that we have: 
\begin{align}
\label{eq:LB}
    LB & \geq \sum_j h_j t_j^* + \sum_j p_j y_j^* + \eta^*
\end{align}

On the other hands, since  $(\lambda^*, z^*, x^*, q^*)$ is an optimal solution to the subproblem in iteration $r$, from (\ref{innermostp}), obviously we have 
\beqn
\label{eq:SPO}
\sum_j P_i q_i^{*}  + \beta \sum_{i,j} d_{i,j} x_{i,j}^* \leq \sum_j P_i q_i  + \beta \sum_{i,j} d_{i,j} x_{i,j},
\eeqn
for \textbf{every} $(x, q)$ that satisfies constraints (\ref{iSPstart})-(\ref{iSPend}) given $\lambda = \lambda^*$ and $z = z^*$. 
Since the extreme point $(\lambda^*, z^*)$ has already been identified and related constraints are added to the MP before iteration $r$, there exists $j \leq r$ where $(x^l, q^l)$ needs to satisfy  constraints (\ref{iSPstart})-(\ref{iSPend}) given $\lambda = \lambda^*$ and $z = z^*$. From (\ref{eq:eta}) and (\ref{eq:SPO}), at iteration $r$ we have:
\beqn
\label{eq:etaa}
\sum_j P_i q_i^{*}  + \beta \sum_{i,j} d_{i,j} x_{i,j}^* \leq \sum_j P_i q_i^l  + \beta \sum_{i,j} d_{i,j} x_{i,j}^l \leq \eta^*.
\eeqn
From (\ref{eq:UB}), (\ref{eq:LB}), and (\ref{eq:etaa}), we have: $LB \geq UB$, which implies $LB = UB$ because LB cannot exceed UB. Thus, any repeated extreme point $\xi^* = (\lambda^*, z^*)$ implies convergence. Since $\Xi^{\sf e}$ is a finite set with $R$ elements, the proposed algorithm is guaranteed to converge in $O(R)$ iterations.

\subsection{KKT-based subproblem reformulation}
\label{kkta}
Similar to the CCG algorithm in \cite{CCG}, we can use KKT conditions to transform the subproblem (\ref{SPP}) into an equivalent MILP problem. The KKT conditions are applied to the innermost problem (\ref{innermostp}). The resulting set of KKT conditions contains stationarity, primal feasibility, dual feasibility,  complementary slackness constraints. Each  non-linear complementary constraint can be transformed into a set of linear constraints by using Fortuny-Amat transformation \cite{BigM}. Next, we will show the derivation to the final reformulation through KKT conditions. Firstly, the Lagrangian function of the inner minimization problem in the subproblem is:
\begin{align*}
    & \mathcal{L}(x_{i,j},q_i,s_i,u_j,\gamma_{i,j}^{1},\gamma_{i}^{2}) = \sum_i P_i q_i + \beta \sum_{i,j} d_{i,j} x_{i,j} \\
    & - \sum_{i} \gamma_{i}^{2} q_i - \sum_{i,j} \gamma_{i,j}^{1} x_{i,j} + \sum_{j} u^{1}_j \big( \sum_{i} x_{i,j} - (1 - z_j) C_j \hat{t}_j  \big) \\
    & \sum_{j} u^{2}_j \big( \sum_{i} x_{i,j} - \hat{y}_j \big)  + \sum_i s_i \big(  \bar{\lambda}_i  + g_i \tilde{\lambda}_i - q_i - \sum_{j}x_{i,j} \big) \\ 
    &+ \sum_{i,j} \pi_{i,j} \big(x_{i,j} - a_{i,j} C_{j} \big) 
\end{align*}
The KKT conditions are:
\begin{subequations}
\begin{align}
    & \frac{\partial \mathcal{L}}{\partial x_{i,j}} = \beta d_{i,j} + u_j + \pi_{i,j} - s_i = \gamma_{i,j}^{1} \geq 0, \forall i,j \label{stationary_w}\\
    & \frac{\partial \mathcal{L}}{\partial q_{i}} = P_i - s_i =  \gamma_{i}^{2} \geq 0, \: \forall i \label{stationary_q}\\
    & \sum_{j} x_{i,j} + q_i \geq \lambda_i, \: \forall i \label{primal_constr1}\\
    & \sum_{i} x_{i,j} \leq C_j \hat{t}_j (1 - z_j); \quad \sum_{j} x_{i,j} \leq \hat{y}_j, \forall j \label{primal_constr2}\\
    & 0 \leq x_{i,j} \leq a_{i,j} C_j, ~ \forall i,j; ~~ q_i \geq 0, ~~ \forall i \label{primal_var}\\
    & \gamma_{i,j}^{1}, \pi_{i,j} \geq 0, \: \forall i,j; \: \gamma_{i}^{2}, s_i \geq 0, \: \forall i; \: u^{\sf 1}_j, u^{\sf 2}_j \geq 0, \: \forall j \label{dual_constr}\\
    & \big(\lambda_i - \sum_{j} x_{i,j} - q_i \big) s_i = 0, \: \forall i \label{comp_constr}\\
    & \big( \sum_{i} x_{i,j} - (1 - z_j) C_j \hat{t}_j \big) u^{1}_j = 0, \: \forall j\\
    & \big( \sum_{i} x_{i,j} - \hat{y}_j \big) u^{2}_j = 0, \: \forall j\\
    & \big( x_{i,j} - a_{i,j} C_{j} \big) \pi_{i,j} = 0, \: \forall i,j \\
    & x_{i,j} \gamma_{i,j}^{1} = 0, \: \forall i,j \quad q_i \gamma_{i}^{2} = 0, \: \forall i\label{comp_var}
\end{align}
\end{subequations}
where (\ref{stationary_w}) and (\ref{stationary_q}) are the stationary conditions, (\ref{primal_constr1}) - (\ref{primal_var}) are the primal feasibility conditions, (\ref{dual_constr}) is dual feasibility condition, and (\ref{comp_constr}) - (\ref{comp_var}) are the complementary slackness conditions. We can then re-write the subproblem with complementary constraints:

\begin{subequations}
\label{comp_reformulation}
\begin{align}
    & \underset{(\lambda, z) \in \Xi, x,q}{\text{max}} \: \sum_i P_i q_i + \beta \sum_{i,j} d_{i,j} x_{i,j} \\
    & \text{s.t.} \quad 0 \leq  P_i d_{i,j} + u_j - s_i \perp  x_{i,j}, \: \forall i,j\\
    & 0 \leq  P_i - s_i \perp q_i \geq 0, \: \forall i \\
    & 0 \leq C_j \hat{t}_j (1 - z_j) - \sum_{i} x_{i,j} \perp u^{1}_j \geq 0,\: \forall j \\
    & 0 \leq \hat{y}_j - \sum_{i} x_{i,j} \perp u^2_j \geq 0, \: \forall j \\
    & 0 \leq \sum_{j} x_{i,j} + q_i - \bar{\lambda}_i - g_i \tilde{\lambda}_i \perp s_i \geq 0, \: \forall i\\
    & 0 \leq a_{i,j} C_j - x_{i,j} \perp \pi_{i,j} \geq 0, \: \forall i,j\\
    & \sum_{i} g_i \leq \Gamma;~ g_i \in [0,1], \: \forall i; \: \sum_{j} z_j \leq K;\: z_j \in \{0,1\}, \forall j, \label{comp_uncertainty}
\end{align}
\end{subequations}
where (\ref{comp_uncertainty}) is the constraint in uncertainty set $\Xi$. Note that a complimentary constraint $ 0 \leq x \perp \gamma \geq 0$ implies a set of constraints including ($x \geq 0, \gamma \geq 0, x \gamma = 0$). Notice that there exists an bilinear term between primal variable $x$ and its dual variable $\gamma$. Fortunately, this non-linear complimentary constraint can be transformed into exact linear constraints by using Fortuny-Amat transformation \cite{BigM}. Let $M$ be the sufficient large number and $v$ be the binary variable. The equivalent transformation of complementary constraint is shown as follows:
\begin{subequations}
\begin{align}
    0 \leq x \leq v M\\
    0 \leq \gamma \leq (1 - v) M
\end{align}
\end{subequations}
By applying this transformation for all complimentary constraints (\ref{comp_constr}) to (\ref{comp_var}), we can obtain the final reformulation of the subproblem (\ref{SPP}) as follows: 

\begin{subequations}
\label{KKT_based}
\begin{align}
& \underset{(\lambda, z) \in \Xi, x, q}{\text{max}} \: \sum_i P_i q_i + \beta \sum_{i,j} d_{i,j} x_{i,j} \\
& \text{s.t.} \quad 0 \leq x_{i,j} \leq M^1_{i,j} (1 - v^1_{i,j}), \forall i,j\\
& 0 \leq  P_i d_{i,j} + \pi_{i,j}  + u^1_j +u^2_j - s_i \leq M^1_{i,j} v^{1}_{i,j}, \: \forall i,j\\
& 0 \leq  P_i - s_i \leq M_{i}^{2} v^{2}_{i}, ~ \forall i \\
& 0 \leq q_i \leq M_{i}^{2} (1 - v^2_i), ~ \forall i \\
& 0 \leq C_j \hat{t}_j (1 - z_j) - \sum_{i} x_{i,j} \leq M_{j}^{3} v^{3}_{j},~ \forall j \\
& 0 \leq u^1_j \leq M_{j}^{3} (1 - v^{3}_{j}), ~ \forall j \\
& 0 \leq \hat{y}_j - \sum_{i} x_{i,j} \leq M_j^4 v^4_j, ~ \forall j\\
& 0 \leq u^2_j \leq M_j^4 (1 - v^4_j), ~ \forall j\\
& 0 \leq \sum_{j} x_{i,j} + q_i - \bar{\lambda}_i - g_i \tilde{\lambda}_i \leq M^{5}_{i} v^{5}_{i}, ~ \forall i\\
& 0 \leq s_i \leq M^{5}_{i} (1 - v^{5}_{i}), ~ \forall i \\
& 0 \leq a_{i,j}C_{i,j} - x_{i,j} \leq M_{i,j}^{6} v^{6}_{i,j}, ~ \forall i,j\\
& 0 \leq \pi_{i,j} \leq M_{i,j}^{6} (1 - v^{6}_{i,j}), ~ \forall i,j\\
& v_{i,j}^1,~ v_{i,j}^6 \in \{0,1\},~\forall i,j \\
& v_i^2, v_i^5 \in \{0,1\},~\forall i; v_j^3,~v_j^4 \in \{0,1\},~\forall j, \label{KKTend}
\end{align}
\end{subequations}
where $M^1_{i,j}, M_{i}^{2}, M_{j}^{3}, M_j^4, M^{5}_{i}$ and $M_{i,j}^{6}$ are sufficiently large numbers.

\subsection{ADR Reformulation}
\label{RC_reformulation}

To transform the ADR model (\ref{ADRp2}) into an MILP, we need to reformulate each robust constraint from (\ref{constr_ADR1}) to (\ref{constr_ADR7}) as  an equivalent set of linear equations. First, consider constraint (\ref{constr_ADR1}). It is easy to see that  (\ref{constr_ADR1}) is equivalent to:
\begin{subequations}
\label{ADRc1}
\begin{align}
& \phi \geq \max_{\mathbf{g} \geq 0,\mathbf{z} \geq 0} \sum_{i \in I}  P_{i} \bigg[ \sum_{e \in I} E_{i}^{e} (\bar{\lambda_e} + g_e \tilde{\lambda}_e) + \sum_{l \in J} F_{i}^l z_{l} + G_i \bigg] \nonumber \\
& + \beta \sum_{i \in I} \sum_{j\in J} d_{i,j} \bigg[ \sum_{e \in I} A_{i,j}^{e} (\bar{\lambda_e} + g_e \tilde{\lambda}_e) + \sum_{l \in J} B_{i,j}^l z_{l} + D_{i,j}\bigg]   \\
&\text{s.t.} \quad \sum_e g_e \leq \Gamma, \quad (\mu^{0}) \label{ADRc11}\\
& g_e \leq 1, ~ \forall e \quad (\eta_{e}^{0})\\
& \sum_{j} z_l \leq K, \quad   \: (v^{0})\\
& z_l \leq 1, ~\forall l, \quad (\sigma_{l}^{0}) \label{ADRc14} 
\end{align}
\end{subequations}
where $\mu^{0}$, $\eta_{e}^{0}$, $v_0$, and $\sigma_{l}^0$ are dual variables for constraints from (\ref{ADRc11}) to (\ref{ADRc14}), respectively.
Be rearranging the terms in  problem (\ref{ADRc1}), we can rewrite  problem (\ref{ADRc1}) as follows:

\begin{subequations}
\label{ADRcon1}
\begin{align}
& \phi \geq  \sum_{i \in I} \sum_{e \in I} P_i E_i^e \bar{\lambda}_e + \sum_{i \in I} P_i G_i + \beta \sum_{i \in I} \sum_{j \in J} \sum_{e \in I} d_{i,j} A_{i,j}^{e} \bar{\lambda}_e \nonumber \\
& + \beta \sum_{i \in I} \sum_{j \in J} d_{i,j} D_{i,j}  + \max_{\mathbf{g} \geq 0,\mathbf{z} \geq 0} \sum_{i \in I} \sum_{l \in J} P_i F_{i}^l z_l \nonumber \\
&  + \sum_{i \in I} \sum_{e \in I}  P_i E_i^e g_e \tilde{\lambda}_e +  \beta \sum_{i \in I} \sum_{j \in J} \sum_{l \in J} d_{i,j} B_{i,j}^l z_{l} \nonumber \\
& + \beta \sum_{i \in I} \sum_{j \in J} \sum_{e \in I} d_{i,j} A_{i,j}^{e} g_e  \tilde{\lambda}_e \\
&\text{s.t.} \quad \sum_e g_e \leq \Gamma, \quad (\mu^{0}) \label{ADRc1s}\\
& g_e \leq 1, ~ \forall e \quad (\eta_{e}^{0})\\
& \sum_{j} z_l \leq K, \quad   \: (v^{0})\\
& z_l \leq 1, ~\forall l. \quad (\sigma_{l}^{0}) \label{ADRc1e} 
\end{align}
\end{subequations}

By LP theory, the dual of the maximization problem on the right hand side of (\ref{ADRcon1}) is the following linear minimization  problem. Thus, problem (\ref{ADRcon1}) (i.e., the robust constraint  (\ref{constr_ADR1})) is equivalent to the following set of linear equations:
\begin{subequations}
\label{ADRconstraint1}
\begin{align}
& \phi - \sum_{i \in I} \sum_{e \in I} P_i E_i^e \bar{\lambda}_e - \sum_{i \in I} P_i G_i - \beta \sum_{i \in I} \sum_{j \in J} \sum_{e \in I} d_{i,j} A_{i,j}^{e} \bar{\lambda}_e  \nonumber \\
& - \beta \sum_{i \in I} \sum_{j \in J} d_{i,j} D_{i,j} -  k v^{0} - \Gamma \mu^{0} - \sum_{e \in I} \eta_{e}^{0} - \sum_{l \in J} \sigma_{l}^{0} \geq 0 \\
& \eta_{e}^{0} + \mu^{0} \geq \beta \sum_{i \in I} \sum_{j \in J} d_{i,j} A_{i,j}^{e} \tilde{\lambda}_{e} + \sum_{i \in I} P_{i} E_{i}^{e} \tilde{\lambda}_{e}, ~ \forall e\\
& \sigma_{l}^{0} + v^{0} \geq \beta \sum_{i \in I} \sum_{j \in J} d_{i,j} B_{i,j}^l + \sum_{i \in I} P_i F_{i}^l,~ \forall l\\
& \eta_{e}^{0}, v^{0}, \mu^{0} \geq 0; \quad  \sigma_{l}^{0} \geq 0, ~\forall l.
\end{align}
\end{subequations}

By following similar procedure, we can reformulate individual robust constraints from (\ref{constr_ADR2}) to (\ref{constr_ADR7}) into equivalent sets of linear equations. For the sake of brevity, we will present the final set of equations for each constraint only. Specifically, (\ref{constr_ADR2}) is equivalent to:
\begin{subequations}
\label{ADRconstraint2}
\begin{align}
& \sum_{i \in I} \sum_{e \in I} A_{i,j}^{e} \bar{\lambda}_e + \sum_{i \in I} D_{i,j} - C_{j}\hat{t}_j + \sum_{e \in I} \eta_{e,j}^{1} + \Gamma \mu_j^{1} \nonumber\\
& + \sum_{l \in J} \sigma_{l,j}^{1} + K v_{j}^{1} \leq 0, \: \forall j\\
& \eta_{e,j}^{1} + \mu_{j}^{1} \geq \sum_{i \in I}  A_{i,j}^{e} \tilde{\lambda}_{e}, \: \forall e, j\\
& \sigma_{l,j}^{1} +  v_{j}^{1} \geq \sum_{i \in I}  B_{i,j}^l, \: \forall l, l \neq j \\
& \sigma_{l,j}^{1} +  v_{j}^{1} \geq \sum_{i \in I}  B_{i,j}^l + C_{j}\hat{t}_j; \: \forall j, l = j \\
& \eta_{e,j}^{1} \geq 0, \: \forall e, j; \: v_{j}^{1}, \mu_{j}^{1} \geq 0, \: \forall j; \:  \sigma_{l,j}^{1} \geq 0, \forall l,j.
\end{align}
\end{subequations}

The equivalent reformulation of  constraint (\ref{constr_ADR3}) is:
\begin{subequations}
\label{ADRconstraint3}
\begin{align}
& \sum_{i \in I} \sum_{e \in I} A_{i,j}^{e} \bar{\lambda}_e + \sum_{i \in I} D_{i,j} - y_j + \sum_{e \in I} \eta_{e,j}^{1} + \Gamma \mu_j^{1} \nonumber\\
& + \sum_{l \in J} \sigma_{l,j}^{1} + K v_{j}^{1} \leq 0, \: \forall j\\
& \eta_{e,j}^{2} + \mu_{j}^{2} \geq \sum_{i \in I}  A_{i,j}^{e} \tilde{\lambda}_{e}, \: \forall e, j\\
& \sigma_{l,j}^{2} +  v_{j}^{2} \geq \sum_{i \in I}  B_{i,j}^l, \: \forall l \\
& \eta_{e,j}^{2} \geq 0, \: \forall e, j; \: v_{j}^{2}, \mu_{j}^{2} \geq 0, \: \forall j; \:  \sigma_{l,j}^{2} \geq 0, \forall l,j.
\end{align}
\end{subequations}

Similarly, constraint (\ref{constr_ADR4}) is equivalent to:
\begin{subequations}
\label{ADRconstraint4}
\begin{align}
& \sum_{j \in J} \sum_{e \in I} A_{i,j}^{e} \bar{\lambda}_e + \sum_{e \in I} E_{i}^{e} \bar{\lambda}_{e} + \sum_{j \in J} D_{i,j} + G_i - \bar{\lambda}_{i}  \nonumber\\
& - \sum_{e \in I} \eta_{i,e}^3 - \Gamma \mu_i^3 - \sum_{l \in J} \sigma_{i,l}^{3} - K v_{i}^{3} \geq 0, \forall i\\
& \eta_{i,e}^{3} + \mu_{i}^3 \geq - \sum_{j \in J} A_{i,j}^{e} \tilde{\lambda}_{e} - E_{i,e} \tilde{\lambda}_{e}, \: \forall i,e \\
& \eta_{i,e}^{3} + \mu_{i}^3 \geq - \sum_{j \in J} A_{i,j}^{e} \tilde{\lambda}_{e} - E_{i,e} \tilde{\lambda}_{e} + \tilde{\lambda}_i, \: \forall i,e = i\\
&  \sigma_{i,l}^{3} + v_{i}^{3} \geq - \sum_{j \in J}  B_{i,j}^l - F_{i}^l, \: \forall i, l \\
& \eta_{i,e}^{3} \geq 0, \: \forall i,e; \: v_{i}^{3}, \mu_{i}^3 \geq 0, \: \forall i; \:  \sigma_{i,l}^{3} \geq 0, \forall i,l.
\end{align}
\end{subequations}

Constraint (\ref{constr_ADR5}) is equivalent to:
\begin{subequations}
\label{ADRconstraint5}
\begin{align}
& \sum_{e \in I} E_{i}^{e} \bar{\lambda}_e + G_i - \sum_{e \in I} \eta_{i,e}^3 - \Gamma \mu_{i}^4 - \sum_{l \in J} \sigma_{i,l}^{4} - K v_{i}^4 \geq 0, ~ \forall i\\
& -\eta_{i,e}^{4}-\mu_{i}^4 \leq E_{i}^{e} \tilde{\lambda}_{e}, ~ \forall i,e\\
& - \sigma_{i,l}^{4} -  v_{i}^4 \leq F_{i}^l, ~ \forall i,l\\
& \eta_{i,e}^{4} \geq 0, ~ \forall i,e; \: v_{i}^4, \mu_{i}^4 \geq 0, ~ \forall i,j; \:  \sigma_{i,l}^{4} \geq 0,~ \forall i,l.
\end{align}
\end{subequations}

Constraint (\ref{constr_ADR6}) is equivalent to:
\begin{subequations}
\label{ADRconstraint6}
\begin{align}
& \sum_{e} A_{i,j}^{e} \bar{\lambda}_e + D_{i,j} -  \sum_{e} \eta_{i,j,e}^5 - \Gamma \mu_{i,j}^5 - \sum_{l} \sigma_{i,j,l}^{5} - K v_{i,j}^5  \nonumber\\
& \geq 0,~ \forall i,j\\
& -\eta_{i,j,e}^{5} - \mu_{i,j}^5 \leq A_{i,j}^{e} \tilde{\lambda}_{e}, ~ \forall i,j,e\\
& - \sigma_{i,j,l}^{5} - v_{i,j}^{5} \leq B_{i,j}^l, ~ \forall i,j,l\\
& \eta_{i,j,e}^{5} \geq 0, ~ \forall i,j,e; \: v_{i,j}^{5}, \mu_{i,j}^5 \geq 0, ~ \forall i,j; \:  \sigma_{i,j,l}^{5} \geq 0, ~\forall i,j,l.
\end{align}
\end{subequations}

Finally,  constraint (\ref{constr_ADR7}) is equivalent to:
\begin{subequations}
\label{ADRconstraint7}
\begin{align}
& \sum_{e} A_{i,j}^{e} \bar{\lambda}_e + D_{i,j} -  \sum_{e} \eta_{i,j,e}^6 - \Gamma \mu_{i,j}^6 - \sum_{l} \sigma_{i,j,l}^{6}  - K v_{i,j}^6 \nonumber\\
&  \leq a_{i,j} C_j, ~ \forall i,j\\
& -\eta_{i,j,e}^{6} - \mu_{i,j}^6 \leq A_{i,j}^{e} \tilde{\lambda}_{e}, ~\forall i,j,e\\
& - \sigma_{i,j,l}^{6} - v_{i,j}^{6} \leq B_{i,j}^l, ~ \forall i,j,l\\
& \eta_{i,j,e}^{6} \geq 0, ~ \forall i,j,e; \: v_{i,j}^{6}, \mu_{i,j}^6 \geq 0, ~ \forall i,j; \:  \sigma_{i,j,l}^{6} \geq 0, ~\forall i,j,l.
\end{align}
\end{subequations}

By replacing the sets of constraints from (\ref{ADRconstraint1}) to (\ref{ADRconstraint7}) into the ADR model (\ref{ADRp2}), we obtain the MILP as shown in (\ref{MILP_ADR}).

\subsection{Deterministic Formulation}
\label{DetermiA}
The deterministic formulation of the service placement and sizing problem is the following MILP problem:

\begin{subequations}
\begin{align}
\label{deterministic}
& \underset{y,t,x,q}{\text{min}}  ~~\sum_j p_j y_j +\sum_j h_j t_j   + \sum_i P_i q_i + \beta \sum_{i,j} d_{i,j} x_{i,j}  \nonumber\\
& \text{s.t.} 
~~ \sum_{j} p_{j} y_{j} + f_j (1 - t_{0})t_{j} \leq B\\
& 0 \leq y_j \leq C_j t_j, ~ \forall j\\
& \sum_{i} x_{i,j}\leq y_j t_j, ~ \forall j \\
& \sum_{j} x_{i,j} + q_i \geq \bar{\lambda}_i,~ \forall i \\
& 0\leq x_{i,j} \leq a_{i,j} C_j, ~\forall i,j \\
& t_j \in \{0,1\}, ~ \forall j;~~ y_{j} \in \mathbb{Z},~\forall j; ~~q_i \geq 0,~\forall i.
\end{align}
\end{subequations}

In the deterministic algorithm, the SP first solves this deterministic problem, using the nominal demand $\bar{\lambda}$ to obtain the optimal service placement and resource procurement solution $(y^*, t^*)$. In the actual operation stage, after observing the actual realization of the demand and EN failures, given the $(y^*, t^*)$ as the input, the SP will solve the problem (\ref{actual}) to find the optimal workload allocation decision $(x, q)$. 

\subsection{Two-stage Stochastic Programming Model}
\label{SOappen}

Instead of optimizing the workload allocation in the worst-case uncertainty realization in the second stage as in the ARO model ($\mathcal{P}_1$), in the stochastic model, the SP aims to optimize the expected cost in the second stage. Let $\xi = \{ \xi^1, \ldots, \xi^n, \ldots, \xi^N\}$ be the set of $N$ scenarios representing the uncertainties of the demand and EN failures, $n$ is the scenario index and $\xi^n = (\lambda^n, z^n)$. Also, denote by $\pi^n$ the probability scenario $n$. The two-stage stochastic edge service placement and workload allocation model is: 
\begin{subequations}
\begin{align}
& \underset{y,t,x,q}{\text{min}}  \sum_{j} p_{j} y_{j} + \sum_j f_j(1 - t_0) t_j \nonumber\\
& +  \: \sum_{n = 1}^{N} \pi_n \bigg\{ \beta \sum_{i} \sum_{j} d_{i,j}^{n} x_{i,j}^{n} + \sum_{i} P_i^{n} q_i^{n} \bigg\} \\
&\text{s.t.} ~~ 0 \leq y_j \leq C_j t_j, \: \forall j \quad \label{SO_constr1} \\
& \sum_j p_{j} y_{j} + \sum_j h_j t_{j} \leq B, \label{SO_constr2}\\
& \sum_{i} x_{i,j}^{n} \leq y_{j} t_{j} (1 - z_j^{n}) , \: \forall j, n \label{SO_constr3}\\
& q_{i}^{n} + \sum_{j} x_{i,j}^{n} \geq \lambda_{i}^{n}, \; \forall i, n \label{SO_constr4}\\
&  t_j \in \{0,1\},\: \forall j; \quad y_{j} \in \mathbb{Z},~\forall j \label{SO_VAR1}\\
&  0 \leq x_{i,j}^{n} \leq a_{i,j} C_j, \forall i,j,n;~~ \: q_i^{n} \geq 0, \:\forall i,n, 
\end{align}
\end{subequations}\\
where $(x^n, q^n)$ expresses the second-stage decision corresponding to scenario $n$. 


\subsection{Computational Time with $\Gamma = 5$}
\label{appendix:gamma5}

Table \ref{table:runtimeGamma5} presents the computational time of the duality-based CCG algorithm and the ADR algorithm with $\Gamma = 5$. In cases where an algorithm fails to produce a result within 10000 seconds, we manually terminate its execution.

\begin{table}[h!]
\centering
\begin{tabular}{|c|c|clll|cll|}
\hline
Problem size                                                                             & K                      & \multicolumn{4}{c|}{Duality-based CCG (s)}                                                                            & \multicolumn{3}{c|}{ADR (s)}     \\ \hline
\multirow{4}{*}{\begin{tabular}[c]{@{}c@{}}I  = 20\\ J = 20\\ ($\Gamma = 5$)\end{tabular}}  & 1                      & \multicolumn{4}{c|}{15.48}                                                                                       & \multicolumn{3}{c|}{76.24}  \\ \cline{2-9} 
                                                                                         & 3                      & \multicolumn{4}{c|}{32.93}                                                                                       & \multicolumn{3}{c|}{82.87}  \\ \cline{2-9} 
                                                                                         & 5                      & \multicolumn{4}{c|}{\begin{tabular}[c]{@{}c@{}}167.87 (0.5\% gap)\\ 64.02 (1\% gap)\end{tabular}}      & \multicolumn{3}{c|}{97.73}  \\ \cline{2-9} 
                                                                                         & 7                      & \multicolumn{4}{c|}{\begin{tabular}[c]{@{}c@{}}172.27 (0.5\% gap)\\ 109.35 (1\% gap)\end{tabular}}  & \multicolumn{3}{c|}{98.12}  \\ \hline
\multirow{4}{*}{\begin{tabular}[c]{@{}c@{}}I = 40\\ J = 40\\ ($\Gamma = 5$)\end{tabular}}   & \multicolumn{1}{l|}{1} & \multicolumn{4}{c|}{309.81}                                                                                      & \multicolumn{3}{c|}{2198}   \\ \cline{2-9} 
                                                                                         & \multicolumn{1}{l|}{3} & \multicolumn{4}{c|}{450.23}                                                                                      & \multicolumn{3}{c|}{2207}   \\ \cline{2-9} 
                                                                                         & \multicolumn{1}{l|}{5} & \multicolumn{4}{c|}{\begin{tabular}[c]{@{}c@{}}5332.72 (0.5\% gap)\\ 865.27(1\% gap)\end{tabular}}   & \multicolumn{3}{c|}{2312}   \\ \cline{2-9} 
                                                                                         & \multicolumn{1}{l|}{7} & \multicolumn{4}{c|}{\begin{tabular}[c]{@{}c@{}}10364 (0.5\% gap)\\ 1753 (1\% gap)\end{tabular}}     & \multicolumn{3}{c|}{2332}   \\ \hline
\multirow{4}{*}{\begin{tabular}[c]{@{}c@{}}I = 100 \\ J = 20\\ ($\Gamma = 5$)\end{tabular}} & 1                      & \multicolumn{4}{c|}{29.89}                                                                                       & \multicolumn{3}{c|}{2839.9} \\ \cline{2-9} 
                                                                                         & 3                      & \multicolumn{4}{c|}{113.24}                                                                                      & \multicolumn{3}{c|}{2953.7} \\ \cline{2-9} 
                                                                                         & 5                      & \multicolumn{4}{c|}{\begin{tabular}[c]{@{}c@{}}479.3 (0.5\% gap)\\ 142.79 (1\% gap)\end{tabular}}   & \multicolumn{3}{c|}{3012.3} \\ \cline{2-9} 
                                                                                         & 7                      & \multicolumn{4}{c|}{\begin{tabular}[c]{@{}c@{}}1298.6 (0.5\% gap)\\  413.24 (1\% gap)\end{tabular}}  & \multicolumn{3}{c|}{3151.1} \\ \hline
\multirow{4}{*}{\begin{tabular}[c]{@{}c@{}}I = 500 \\ J = 30\\ ($\Gamma = 5$)\end{tabular}} & 1                      & \multicolumn{4}{c|}{75.54}                                                                                       & \multicolumn{3}{c|}{NA}       \\ \cline{2-9} 
                                                                                         & 3                      & \multicolumn{4}{c|}{249.59}                                                                                      & \multicolumn{3}{c|}{NA}       \\ \cline{2-9} 
                                                                                         & 5                      & \multicolumn{4}{c|}{\begin{tabular}[c]{@{}c@{}}3979.3 (0.5\% gap)\\ 407.15 (1\% gap)\end{tabular}}  & \multicolumn{3}{c|}{NA}       \\ \cline{2-9} 
                                                                                         & 7                      & \multicolumn{4}{c|}{\begin{tabular}[c]{@{}c@{}}7332.5 (0.5\% gap)\\ 562.79 (1\% gap)\end{tabular}}  & \multicolumn{3}{c|}{NA}       \\ \hline
\multirow{4}{*}{\begin{tabular}[c]{@{}c@{}}I = 1000\\ J = 50\\ ($\Gamma = 5$)\end{tabular}} & \multicolumn{1}{l|}{1} & \multicolumn{4}{c|}{562.15}                                                                                      & \multicolumn{3}{c|}{NA}       \\ \cline{2-9} 
                                                                                         & \multicolumn{1}{l|}{3} & \multicolumn{4}{c|}{757.65}                                                                                      & \multicolumn{3}{c|}{NA}       \\ \cline{2-9} 
                                                                                         & \multicolumn{1}{l|}{5} & \multicolumn{4}{c|}{\begin{tabular}[c]{@{}c@{}}8547.72 (0.5\% gap)\\ 1435.4 (1\% gap)\end{tabular}} & \multicolumn{3}{c|}{NA}       \\ \cline{2-9} 
                                                                                         & \multicolumn{1}{l|}{7} & \multicolumn{4}{c|}{\begin{tabular}[c]{@{}c@{}} NA (0.5\% gap)\\ 2607.6 (1\% gap)\end{tabular}}        & \multicolumn{3}{c|}{NA}       \\ \hline
\end{tabular}
\caption{Run-time experiments}
\label{table:runtimeGamma5}
\end{table}

\end{document}